\documentclass[12pt,leqno,twoside]{amsart}
\usepackage{amsmath,amscd,amssymb,amsfonts,latexsym,mathrsfs,bm,extarrows,graphicx,stackrel}
\usepackage[all,cmtip]{xy}
\usepackage{ulem}
\usepackage{tikz}

\usepackage{srcltx}
\usepackage[colorlinks=true, pdfstartview=FitV, linkcolor=blue, citecolor=blue, urlcolor=blue]{hyperref}
\usepackage{subcaption} 
\usepackage{amsbsy}
\usepackage{amsthm,alltt,dsfont}

\usepackage[margin=1.2in]{geometry}

\usepackage[utf8]{inputenc}
\usepackage{amsthm,soul}
\usepackage{enumitem}
\usepackage{setspace,kantlipsum}
\usepackage{tikz-cd}
\usepackage{textcomp}

\usepackage{colonequals}
\usepackage{mathtools,comment}

\usepackage{graphicx,bm}
\usepackage{mathptmx}

\theoremstyle{plain} 
\newtheorem{thm}{Theorem}[section]
\newtheorem{lemma}[thm]{Lemma}
\newtheorem{proposition}[thm]{Proposition}
\newtheorem{corollary}[thm]{Corollary}

\theoremstyle{definition}

\newtheorem{definition}[thm]{Definition}

\newtheorem{example}[thm]{Example}
\newtheorem{remark}[thm]{Remark}

\newtheorem{conjecture}[thm]{Conjecture}

\newcommand{\Z}{\mathbb{Z}}
\newcommand{\Q}{\mathbb{Q}}

\newcommand{\cD}{\mathcal{D}}

\newcommand{\cF}{\mathcal{F}}
\newcommand{\cG}{\mathcal{G}}
\newcommand{\cH}{\mathcal{H}}

\newcommand{\cL}{\mathcal{L}}
\newcommand{\cM}{\mathcal{M}}

\newcommand{\cO}{\mathcal{O}}

\newcommand{\cS}{\mathcal{S}}

\newcommand{\cV}{\mathcal{V}}

\newcommand{\bC}{\mathbf{C}}
\newcommand{\bD}{\mathbf{D}}

\newcommand{\bK}{\mathbf{K}}
\newcommand{\bL}{\mathbf{L}}

\newcommand{\bQ}{\mathbf{Q}}
\newcommand{\bR}{\mathbf{R}}

\newcommand{\bV}{\mathbf{V}}

\newcommand{\bZ}{\mathbf{Z}}

\newcommand{\sD}{\mathscr{D}}

\newcommand{\Gr}{\mathrm{Gr}}
\newcommand{\gr}{\mathrm{Gr}}
\newcommand{\sd}{\mathrm{sd}}
\newcommand{\Pol}{\mathrm{Pol}}

\newcommand{\DR}{\mathrm{DR}}

\newcommand{\MHM}{\mathrm{MHM}}
\newcommand{\MH}{\mathrm{MH}}

\DeclareMathOperator{\rk}{rank}
\DeclareMathOperator{\im}{Image}

\newcommand{\Hom}{\mathrm{Hom}}

\newcommand{\mhm}{\mathrm{MHM}}
\newcommand{\td}{\mathrm{td}}
\newcommand{\MHS}{\mathrm{mHs}}
\newcommand{\rat}{{\it {rat}}}

\begin{document}

\title[Class version of the Hodge index theorem]{Around a class version of the Hodge index theorem  for singular varieties}


\author{Mohammadali Aligholi}
\address{M. Aligholi : Department of Mathematics, University of Wisconsin-Madison, 480 Lincoln Drive, Madison WI 53706-1388, USA}
\email{aligholi@wisc.edu}

\author{Lauren\c{t}iu Maxim}
\address{L. Maxim : Department of Mathematics, University of Wisconsin-Madison, 480 Lincoln Drive, Madison WI 53706-1388, USA, \newline
{\text and} \newline Institute of Mathematics of the Romanian Academy, P.O. Box 1-764, 70700 Bucharest, ROMANIA}
\email{maxim@math.wisc.edu}

\author{J\"{o}rg Sch\"{u}rmann}
\address{J. Sch\"{u}rmann : Mathematisches Institut, Universti\"{a}t M\"{u}nster, Einsteinstr. 62, 48149 M\"{u}nster, Germany}
\email{jschuerm@uni-muenster.de}

\begin{abstract}
We give an overview of recent developments around a characteristic class version of the Hodge index theorem for singular complex algebraic varieties. This was formulated by Brasselet-Sch\"urmann-Yokura as a conjecture expressing the Goresky-MacPherson homology $L$-classes in terms of suitable Hodge-theoretic $L$-classes. Along the way, we clarify the relationship between several notions of  $L$-classes appearing in the literature, but we also include many new cases for which the conjecture is true, e.g., all compact toric varieties, all (matroid) Schubert varieties, 
all Richardson and intersection varieties,
all projective simply connected spherical varieties, and all compact complex algebraic surfaces and threefolds. 
\end{abstract}




\maketitle

\tableofcontents

\section{Introduction}
The classical Hodge index theorem $\sigma(X)=\chi_1(X)$   expresses the signature $\sigma(X)$ of a compact complex manifold $X$ as the value at $y=1$ of the Hirzebruch $\chi_y$-genus $\chi_y(X)$; see \cite{H}. In the complex algebraic context, this computes the signature of $X$ in terms of its Hodge numbers,
based on the hard Lefschetz theorem and the corresponding Hodge-Riemann bilinear relations. \\

The Hodge index theorem has been extended to the singular context 
by using the signature and, respectively, intersection cohomology $\chi_y$-genus, $I\chi_y$, defined via the Goresky-MacPherson (middle-perversity)  intersection cohomology groups \cite{GM}. Specifically, for 
a (possibly singular) compact complex algebraic variety $X$ of pure dimension, the signature $\sigma(X)$ of the Poincar\'e duality pairing on the middle intersection cohomology group can be computed as 
$$\sigma(X) = I\chi_1(X)\:,$$ i.e., as a certain expression in the Hodge numbers of the pure Hodge structure on the intersection cohomology groups of $X$ (defined via M. Saito's theory of pure Hodge modules); see \cite{MSS0,FPS}.\\

The signature $\sigma(X)$ and the intersection cohomology Hirzebruch genus $I\chi_y(X)$ of a pure-dimensional compact complex algebraic variety are the degrees of the Goresky-MacPherson homology $L$-class $L_*(X)$ \cite{GM} and, respectively, of the intersection cohomology Hirzebruch class $IT_{y*}(X)$ \cite{BSY}, and it was conjectured in \cite{BSY} that a corresponding 
characteristic class version  of the Hodge index theorem should hold. 
More precisely, one has the following:
\begin{conjecture}[Characteristic class version of Hodge index theorem] \label{BSYi}
Let $X$ be a compact pure-dimensional complex algebraic variety. Then
\begin{equation}\label{conji}
L_*(X)=IT_{1*}(X).
\end{equation}
\end{conjecture}

In fact, the conjecture seems to be attributed to Cappell and Shaneson (see \cite[Remark 5.4]{BSY} and the references therein), so it predates the construction of Hirzebruch classes of \cite{BSY}. Only the case for rational homology manifolds is explicitly conjectured in \cite{BSY}. 
The conjecture holds in the smooth context by work of Hirzebruch on the generalized Hirzebruch-Riemann-Roch theorem \cite{H}, and it was recently proved in the case of rational homology manifolds in \cite{FPS}, see also \cite{FP}. Several special cases have also been checked in \cite{CMSS} for certain complex hypersurfaces with isolated singularities, in \cite{CMSS2} for global quotients of projective varieties by finite groups, in \cite{MS1} for projective simplicial toric varieties, in \cite{BSW} for Schubert varieties in a Grassmannian; see Section \ref{known} for a more detailed reference list.\\

The purpose of this paper is to explain the motivation for Conjecture \ref{BSYi}, survey its status and recent developments, but to also present some new cases where the conjecture holds, e.g., for all compact toric varieties, all (matroid) Schubert varieties, all Richardson and intersection varieties,
all projective simply connected spherical varieties, and all compact complex algebraic surfaces and threefolds. 
Although some of these cases are perhaps known to experts, we could not find them explicitly stated in the existing literature.

\medskip

Conjecture \ref{BSYi} can already be explained via the following diagram of transformations for $X$ a compact complex algebraic variety (here the compactness is only needed for the $L$-class transformations): 

\begin{equation}\label{diag1}
    \begin{tikzcd}
 K_0(Var/X)   \arrow{r}{\chi_{Hdg}} \arrow[swap]{rdd}{\sd_\bR^{pair}} \arrow{dr}{\sd_\bR } \arrow[bend left=30,rrrr,"T_{y*}"] &  K_{0}(MHM(X)) \arrow{rrr}{MHT_{y*}}[swap]{\boxed{?}}\arrow{d}{\Pol} & &
& H_*(X)\otimes \bQ[y^{\pm 1}] \arrow{d}{y=1}  \\
     & \Omega_\bR(X) \arrow{rrr}{L_*} \arrow{d}{\tau} & & &
     H_{*}(X)\otimes\bQ  \\
     & \Omega^{pair}_\bR(X) \arrow[swap]{rrru}{L^{pair}_*}
     \end{tikzcd}
\end{equation}
with $H_*(X)$ denoting the even-degree Borel-Moore
homology $H^{BM}_{2*}(X;\bZ)$ of the complex algebraic variety $X$.
For $X$ compact this agrees with the usual even degree homology 
$H_{2*}(X;\bZ)$.\\

These transformations and their construction will be explained in Section \ref{sec:sing} below. Let us just say here that all possible transformations  $K_0(Var/X)  \to H_{*}(X)\otimes\bQ $ in the diagram  are the same, and all subdiagrams are known to commute, except the square denoted by $\boxed{?}$. 
For instance, the top part of the diagram commutes by \cite{BSY}, i.e., \[
MHT_{y*} \circ \chi_{Hdg}=T_{y*}.
\]
Moreover, by \cite{BSY},
\[
T_{1*}=L_* \circ \sd_{\bR},
\]
which uses the {\it weak factorization theorem}  behind Bittner's  {\it blow-up relation}
for $K_0(Var/X)$. Here $L_*$ is a functorial $L$-class transformation, so that $L_*([IC_X])=L_*(X)$ is the Goresky-MacPherson $L$-class of $X$ \cite{GM}. In \cite{BSY},  an $L$-class transformation of Cappell-Shaneson \cite{CS} was used, but its construction needs to be corrected by using a {\it shifted version} as indicated later on in Section \ref{Lclass}. This will give a functorial Cappell-Shaneson $L$-class transformation 
  $L_*$  on the cobordism group $ \Omega_\bR(X)$ of {\it self-dual} sheaf complexes, which fits 
with the  Goresky-MacPherson $L$-class. Furthermore, in Theorem \ref{comparison-smooth} we will see that the 
square denoted by $\boxed{?}$ is {\it commutative} for any polarizable variation of pure Hodge structures on a smooth variety $X$ (corresponding to a shifted pure smooth Hodge module.)

Finally, the transformation $\Pol$ was only recently defined in \cite{FPS} as a functorial transformation $$\Pol^{pair}:K_0(\MHM(X)) \to \Omega^{pair}_\bR(X)$$ such that
\[
\sd_\bR^{pair}= \Pol^{pair} \circ \chi_{Hdg}, 
\]
with $\Pol^{pair}:=\tau \circ \Pol$ for $\tau:\Omega_\bR(X)\to \Omega^{pair}_\bR(X) $ a canonical quotient transformation.
However, as explained later on, the construction of $\Pol$ can be lifted to $\Omega_\bR(X)$ such that
\[
\sd_\bR=\Pol \circ \chi_{Hdg}.
\]
Another possible functorial $L$-class transformation  $L^{pair}_*$
defined for a cobordism group of self-dual constructible sheaf complexes in terms of {\it pairings} was recently introduced in \cite{FPS2}.
As we will show, both of these $L$-class transformations are compatible, in the sense that
$$L^{pair}_*\circ \tau = L_*\:.$$

By construction, one has for $X$ pure-dimensional
\[\Pol(IC'^H_X)=[IC_X] \in \Omega_\bR(X),\]
where $IC_X^H$ is the intersection cohomology Hodge module and \[ IC'^H_X:=IC_X^H[-\dim(X)].\]
Here, $IC_X$ is endowed with the canonical self-duality pairing. 
Conjecture \ref{BSYi} amounts to showing that the two ways of mapping $[IC'^H_X]$ around the middle square $\boxed{?}$ in the diagram coincide, i.e.,
\[(L_* \circ \Pol)([IC'^H_X])= MHT_{1*}([IC'^H_X]).\]

\begin{remark}\label{r11}
If $[M] \in \im(\chi_{Hdg}) \subset K_0(\MHM(X))$, then the commutativity of the relevant parts of diagram \eqref{diag1} as explained above, yield that 
\[ (L_* \circ \Pol)([M])= MHT_{1*}([M]).\]
This applies in particular to the constant Hodge complex $[\bQ^H_X]=\chi_{Hdg}([id_X])$. This also implies that Conjecture \ref{BSYi} will hold whenever $[IC_X'^H]\in \im(\chi_{Hdg})$, as in the following examples.
\end{remark}

\begin{example}\label{ex12}
    If $X$ is a rational homology manifold, then $\chi_{Hdg}([id_X:X \to X])=[\bQ_X^H]=[IC'^H_X] $, which puts the main results of \cite{FP,FPS} into the framework of this paper.
\end{example}

\begin{example}\label{ex13}
    If $f:Y \to X$ is a small resolution of singularities, then $Rf_*\bQ_Y=IC'_X$, hence $[IC'^H_X]\in \im(\chi_{Hdg}),$ since $\chi_{Hdg}([f:Y \to X]):=[f_!\bQ_Y^H]=[f_*\bQ_Y^H]=[IC'^H_X].$
\end{example}

\begin{example}\label{ex14} 
 Recall that if $f:Y\to X$ is the normalization of $X$, then $Rf_*IC'_Y=f_*IC'_Y=IC'_X$. If, moreover, $X=\cup_i X_i$ is a simple normal crossing divisor in a smooth variety, then its normalization, $Y=\sqcup_i X_i$, is the disjoint union of its smooth irreducible components, hence $IC'_Y=\oplus_i \bQ_{X_i}$. This implies $IC_X'^H \in \im(\chi_{Hdg})$.
\end{example}

\begin{remark}\label{r14}
    More generally, if $f:Y\to X$ is a proper homologically small map of degree one in the sense of Goresky-MacPherson \cite[Section 6.2]{GM2}, then $Rf_*IC'_Y=IC'_X$, hence $[IC'^H_X]\in \im(\chi_{Hdg})$ if $[IC'^H_Y]\in \im(\chi_{Hdg})$, e.g., if $Y$ is a rational homology manifold.
\end{remark}

\begin{remark}\label{rem-submodule}
Conjecture \ref{BSYi} will hold whenever $[IC_X'^H]$ belongs to 
the $K_0(\MHM(pt))$-sub\-module generated by $\im(\chi_{Hdg}) \subset K_0(\MHM(X))$, see Corollary \ref{submodule}.
\end{remark}

\begin{example}[Secant varieties of a smooth projective curve]
    Let $C$ be a smooth projective curve in $\bC P^N$, embedded  by a line bundle which separates $2k$ points. Consider the secant variety $Sec^k(C)$ of $k$-planes through the
curve $C$, the $k$-th secant bundle $B^k(C)$ of $C$, and the natural map $\beta_k:B^k(C) \to Sec^k(C)$. Then $B^1(C)$ is isomorphic to $C$ and the map $\beta_1$ is an isomorphism. Moreover, for $k \geq 2$, $Sec^k(C)$ is a proper subvariety of $\bC P^N$, which is singular along $Sec^{k-1}(C)$, so that $\beta_k$ is a semi-small resolution of singularities. 
Using the fact that the fibers of $\beta_k$ (which are symmetric products of $C$) are irreducible, the decomposition theorem for such semi-small maps yields that 
\[ (\beta_k)_*\bQ^H_{B^k(C)}[2k-1] \simeq \bigoplus_{i=1}^k IC^H_{Sec^i(C)}(i-k),\]
e.g., see \cite[Section 2.4]{Br} and the references therein for more details. It thus follows by induction that  $IC_{Sec^k(C)}'^H$ belong to the $K_0(\MHM(pt))$-sub\-module generated by $\im(\chi_{Hdg})\subset K_0(\MHM(Sec^k(C)))$, hence Conjecture \ref{BSYi} holds for such secant varieties embedded into a projective space via a
suﬃciently positive line bundle. Let us finally note that if $C$ is a rational normal curve, then in fact each secant variety $Sec^k(C)$ is a rational homology manifold. (We thank Bradley Dirks, Sebastian Olano and Debaditya Raychaudhury for bringing the paper \cite{Br} to our attention.)
\end{example}

\begin{example}\label{ex16}
    Assume the pure-dimensional compact complex algebraic variety $Y$ has only ($\bQ$-homologically) isolated singularities. Then, by the arguments of Example \ref{iso}, $[IC_Y'^H]$ 
    belongs to the $K_0(\MHM(pt))$-submodule generated by $\im(\chi_{Hdg}) \subset K_0(\MHM(X))$.
     In particular, if $X$ is a compact complex algebraic surface, its normalization $Y$ has at most isolated singularities, hence it follows by Remark \ref{r14} that $[IC_X'^H]$
     belongs to the $K_0(\MHM(pt))$-submodule generated by $\im(\chi_{Hdg}) \subset K_0(\MHM(X))$.
     In particular, by Remark \ref{rem-submodule}, Conjecture \ref{BSYi} holds for any compact complex algebraic surface.
\end{example}

\begin{remark}\label{rem-3dim}
By a different type of argument, it can be shown that 
Conjecture \ref{BSYi} also holds for any compact complex algebraic variety $X$ of pure dimension, whose singular locus $X_{sing}$ (or non-rational homology locus $\Sigma_X$) is one-dimensional (see Corollary \ref{cor-1dim}).
In particular, the conjecture holds (via normalization) for any compact complex algebraic variety $X$ of pure dimension three
(see Example \ref{ex-1dim}).
\end{remark}

\begin{remark}\label{problem}
It is a very difficult problem, first raised in \cite{S2}, to decide whether for any pure-dimensional complex algebraic variety $X$ the class
 $[IC_X'^H]$ belongs to $\im(\chi_{Hdg})$ or to the $K_0(\MHM(pt))$-submodule generated by $\im(\chi_{Hdg}) \subset K_0(\MHM(X))$ (see also the discussion
 in \cite[Section 10]{FPS2}), so that  Conjecture \ref{BSYi} would be solved by our results. One of the aims of this paper is to show that this indeed is the case for many interesting examples.
\end{remark}

The paper is organized as follows. In Section \ref{charsm} we review the relevant background on characteristic classes in the smooth setting. Section \ref{sec:sing} is devoted to characteristic classes in the singular context. In Subsection \ref{ccc} we introduce the (motivic and, resp., Hodge-theoretic) Hirzebruch classes of \cite{BSY} for singular varieties. In Subsection \ref{cobg} we recall two notions of a cobordism group of self-dual constructible complexes and their relation to mixed Hodge modules (as recently described in \cite{FPS}). In particular we explain the definition of the new transformation $\Pol$ from \cite{FPS}, whose very deep and difficult functoriality (recalled in Theorem \ref{thpol}) underlies almost all results of this paper.
In Section \ref{Lclass} we explain our new  definition of the functorial $L$-class transformation $L_*$, and clarify its  relationship with another notion of topological $L$-class transformation $L_*^{pair}$ introduced in \cite{FPS2}. Subsection \ref{known} collects known positive results about Conjecture \ref{BSY}.
Finally, in Section \ref{secnew} we describe new situations that guarantee that $[IC_X'^H]$ belongs to 
the $K_0(\MHM(pt))$-sub\-module generated by $\im(\chi_{Hdg}) \subset K_0(\MHM(X))$, which (in view of the results of \cite{FPS}) proves Conjecture \ref{BSYi} in these cases. Moreover, in
Subsection \ref{subsec-AM} we show the property $L_*\circ \Pol(-)= MHT_{1*}(-)$
for any polarizable variation of pure Hodge structures on  a smooth variety $X$
(corresponding to a shifted pure smooth Hodge module).

\medskip

{\bf Acknowledgments.} L. Maxim acknowledges support from the Simons Foundation, and from the project ``Singularities and
Applications'' - CF 132/31.07.2023 funded by the European Union - NextGenerationEU
- through Romania's National Recovery and Resilience Plan. J. Sch\"urmann was funded
by the Deutsche
Forschungsgemeinschaft (DFG, German Research Foundation) Project-ID 427320536 – SFB
1442, as well as under Germany's Excellence Strategy EXC 2044 390685587, Mathematics
M\"unster: Dynamics – Geometry – Structure.

\section{Characteristic classes of smooth  complex varieties}\label{charsm}
In this section, we recall the definition of the Hirzebruch class of a smooth
complex algebraic variety, and explain its relation to the $L$-class.

If $X$ is a smooth
complex algebraic variety, the signature (for $X$ also compact) and the
$L$-classes of $X$ are special cases of more general Hodge theoretic
invariants encoded by the Hirzebruch characteristic class (also
called ``the generalized Todd class") $$T_y^*(X):=T_y^*(T_X)$$ of the (holomorphic) tangent
bundle of $X$ (cf. \cite{H}). This is defined by the normalized
power series
\begin{equation}
Q_y(\alpha)=\frac{\alpha(1+y)}{1-e^{-\alpha(1+y)}}-\alpha y \in
\Q[y][[\alpha]],
\end{equation}
that is, \begin{equation}\label{Hsm}
T_y^*(T_X)=\prod_{i=1}^{\dim(X)}Q_y(\alpha_i),
\end{equation}
where $\{\alpha_i\}$ are the Chern roots of the tangent bundle
$T_X$. Note that $Q_{y}(\alpha)$ satisfies 
\begin{displaymath}
Q_{y}(\alpha) = 
\begin{cases}
\:1+\alpha &\text{for $y=-1$,}\\
\:\frac{\alpha}{1-e^{-\alpha}} &\text{for $y=0$,}\\
\:\frac{\alpha}{\tanh \alpha} &\text{for $y=1$.}
\end{cases} 
\end{displaymath}
Therefore, the Hirzebruch
class $T_y^*(T_X)$ coincides with the total Chern class $c^*(T_X)$
if $y=-1$, with the total Todd class $\td^*(T_X)$ if $y=0$, and 
with the total Thom-Hirzebruch $L$-class $L^*(T_X)$ if $y=1$. As usual, cohomology classes of $X$ correspond to those defined via its tangent bundle, e.g., $L^*(X)=L^*(T_X)$.\\

For $X$ compact, the Hirzebruch class appears in the generalized
Hirzebruch-Riemann-Roch theorem (cf. \cite{H}, \S 21.3), which
in particular implies that 
the Hirzebruch genus
\begin{equation}\label{Hichi}
\chi_y(X):=\chi_y(X,\mathcal{O}_X):=\sum_{p \geq 0} \chi(X,\Omega^p_X) \cdot y^p=\sum_{p \geq 0} \left( \sum_{i \geq 0} (-1)^i \dim H^i(X,\Omega^p_X) \right) \cdot y^p
\end{equation} is computed by
\begin{equation}\label{Hchi}\chi_y(X)=\langle {T}_y^*(T_X) , [X] \rangle,\end{equation}
with $\langle -, - \rangle$ the usual Kronecker pairing.
Hence, for the special values $y=-1, 0, 1$, formula \eqref{Hchi} yields the identifications
 \begin{equation}\label{chi-1}\chi_{-1}(X)=\langle c^*(T_X) , [X] \rangle,\end{equation}
 \begin{equation}\label{chi0}\chi_0(X)=\langle \td^*(T_X) , [X] \rangle,\end{equation}
 \begin{equation}\label{chi1}\chi_1(X)=\langle L^*(T_X) , [X] \rangle.\end{equation}
In view of the Gauss–Bonnet–Chern Theorem (which expresses the Euler characteristic $\chi(X)$ as the degree of the total Chern class of $T_X$), the Riemann-Roch Theorem (expressing the arithmetic genus $\chi_a(X)$ as the degree of the total Todd class), and, resp., the Hirzebruch Signature Theorem (expressing the signature $\sigma(X)$ as the degree of the total $L$-class), one then has that
 \begin{equation}\label{sp}
 \chi(X)=\chi_{-1}(X), \ \ \chi_a(X)=\chi_0(X), \ \ \sigma(X)=\chi_1(X).
 \end{equation}

Note that if $h^{p,q}(X):=\dim H^q(X, \Omega^p_X)$ are the Hodge numbers of the smooth compact complex algebraic variety $X$, then 
it follows from \eqref{Hichi} that
\begin{equation}\label{chiy}\chi_y(X)=\sum_{p,q} (-1)^q h^{p,q}(X) \cdot y^p.
\end{equation}
The identification $\sigma=\chi_1$ is usually referred as the (classical) ``Hodge index theorem'', see \cite[Theorem 15.8.2]{H} for the K\"ahler case
(based on the hard Lefschetz theorem and the corresponding Hodge-Riemann bilinear relations),
and the discussion in the bibliographical note of \cite[Section 23.1]{H} for the compact complex manifold case (based on the Atiyah-Singer index theorem).


\section{Characteristic classes of singular complex algebraic varieties}\label{sec:sing}
The above-mentioned cohomology characteristic classes of complex algebraic manifolds have been extended to the singular setting via suitable natural transformations valued in homology, see \cite{BFM, BSY, CS, GM, MP}. Since here we are only interested in the relation between Hirzebruch classes and $L$-classes, we focus on these.

\subsection{Hirzebruch classes}\label{ccc}
Let $X$ be a reduced complex algebraic variety, and we denote by $H_*(X)$  the even-degree Borel-Moore homology $H_{2*}^{BM}(X;\bZ)$.

Let $K_{0}(var/X)$ be the relative Grothendieck group of algebraic
varieties over $X$, i.e., the quotient of the free abelian group of isomorphism
classes of algebraic morphisms $Y\to X$ to $X$, modulo the additivity
relation generated by  the ``scissor relation''
\begin{equation} \label{eq:add}
[Y\to X] = [Z\hookrightarrow Y \to X] + [Y\backslash Z \hookrightarrow Y \to X] 
\end{equation}
for $Z\subset Y$ a closed algebraic subset of $Y$.
By resolution of singularities, $K_{0}(var/X)$ is generated by
classes $[Y\to X]$ with $Y$ smooth pure dimensional and proper
(or projective) over $X$. There is a notion of pushforward for any $f:X'\to X$ defined by 
\[f_{!}([h:Z\to X'])=[f\circ h: Z\to X].\]
In \cite[Theorem 2.1]{BSY}, it was shown that there exists a unique group homomorphism 
\[T_{y*}: K_{0}(var/X)\to H_{*}(X)\otimes\bQ[y] \:,\]
commuting with pushforward
for proper maps, and
satisfying the normalization condition
\[T_{y*}([id_{X}])=T^*_y(T_X) \cap [X]\]
for $X$ smooth and pure-dimensional.

The transformation $T_{y*}$ has a Hodge-theoretic incarnation $MHT_{y*}$ defined as follows. 
For a complex algebraic variety $X$, we let $\MHM(X)$ denote the abelian category of algebraic mixed Hodge modules on $X$, cf. \cite{Saito90}. First, we introduce a K-theoretic version of the Chern class transformation, namely
\[
 \DR_y: K_0(\mhm(X)) \to  K_0(X)[y,y^{-1}],
\]
where $K_0(X)=K_0(Coh(X))$ is the Grothendieck group of coherent sheaves
of $\cO_X$-modules on $X$. 
If $X$ is smooth and $M\in \mhm(X)$, let $(\cM,F_{\bullet}\cM)$ be the underlying filtered $\sD_X$-module of $M$, and set
\[
\begin{split}
\DR_y[M] &\colonequals \sum_{p}  \left[\gr^F_{-p}\DR(\cM)\right] \cdot(-y)^p  \\
&= \sum_{p,i} (-1)^i \left[\cH^i \gr^F_{-p}\DR(\cM)\right] \cdot(-y)^p \in K_0(X)[y,y^{-1}],
\end{split}
\]
with $\gr^F_{-p}\DR(\cM)$ the graded parts of the de Rham complex of $\cM$ with respect to the induced Hodge filtration. Note that only finitely many cohomology groups $\cH^i \gr^F_{-p}\DR(\cM)$ are non-zero.
The definition of $\DR_y$ extends to the case when $X$ is singular by using locally defined closed embeddings into smooth varieties, since the graded quotient cohomology sheaves $\cH^i\gr^F_{p}\DR(\cM)$ are independent of local embeddings and are $\cO_X$-modules. Furthermore, it can also be extended to complexes $M^\bullet \in D^b\mhm(X)$ by applying it to each cohomology module $H^iM^\bullet\in \mhm(X)$, $i \in \Z$, and taking the alternating sum.

\begin{definition}\cite{BSY,S}
The Hodge-theoretic {\it Hirzebruch class transformation}  is defined by the composition
\begin{align*}
MHT_{y\ast}: K_0(\mhm(X)) &\to H_*(X) \otimes \bQ[y^{\pm 1}, (1+y)^{-1}] \\
[M^\bullet] &\mapsto \td_{(1+y)\ast}  \circ \DR_y[M^\bullet],
\end{align*}
where $\td_{(1+y)\ast}$
is the scalar extension of the
Baum--Fulton--MacPherson Todd class transformation $\td_*:K_0(X) \to H_\ast(X) \otimes \bQ$ \cite{BFM}, together with the multiplication by $(1+y)^{-k}$ on $H_k(X)$.
\end{definition}

\begin{remark}
In fact, by \cite[Proposition 5.21]{S}, one has that $$MHT_{y\ast}([M^\bullet])\in H_*(X) \otimes \bQ[y^{\pm 1}],$$ that is, no negative powers of $(1+y)$ appear in the above definition.
\end{remark}

\begin{remark}\label{rem23}
Over a point space, the transformation $ MHT_{y\ast}$ coincides with the Hodge polynomial ring homomorphism $\chi_y:K_0(\MHS^p) \to \bZ[y^{\pm 1}]$ defined on the Grothendieck group of (graded) polarizable mixed Hodge structures by
$$\chi_y([H]):=\sum_p \dim \gr^p_F(H \otimes \bC) \cdot (-y)^p,$$
for $F$ the (decreasing) Hodge filtration of $H \in \MHS^p$. Here we use the well-known equivalence of categories $\MHM(pt) \simeq \MHS^p$.
\end{remark}

The transformations $T_{y*}$ and $MHT_{y*}$ are related to each other via the natural group homomorphism
\[ \chi_{Hdg}:K_0(var/X) \to K_0(\mhm(X)), \ \ [f:Y \to X] \mapsto [f_!\bQ^H_Y] ,\]
with 
$\bQ^H_Y$ denoting the constant mixed Hodge complex on $Y$. Then, as explained in \cite[Section 4]{BSY}, one has that
\begin{equation}\label{compTy}
T_{y*}=MHT_{y*} \circ \chi_{Hdg}.
\end{equation}
Note that $\chi_{Hdg}([id_X:X \to X])=[\bQ^H_X].$

\begin{definition}
The {\it motivic Hirzebruch class} ${T_y}_*(X)$ of a complex algebraic
variety $X$ is defined by 
\begin{equation}
T_{y*}(X):=T_{y*}([id_X])=MHT_{y*}([\bQ_X^H]) \in H_{*}(X)\otimes\bQ[y].
\end{equation}
Similarly, for a pure-dimensional complex algebraic variety $X$, we define
\begin{equation}
IT_{y*}(X):=MHT_{y*}([IC'^H_X]) \in H_*(X) \otimes \bQ[y^{\pm 1}],
\end{equation}
with $IC'^H_X:=IC_X^H[-\dim(X)]$. This is usually referred to the as the {\it intersection cohomology Hirzebruch class} of $X$, though it is still valued in homology.
\end{definition}

For a compact complex algebraic variety $X$ with $a_X:X \to pt$ the
constant map to a point, by using Remark \ref{rem23} it follows that the pushforward $a_{X*}T_{y*}(X)$ is the {\it Hodge polynomial}
\begin{equation}\label{defchi}\chi_y(X):=\chi_y([H^*(X;\bQ)])=\sum_{i,p}
(-1)^i \dim_{\bC} \Gr^p_F H^i(X;\bC) \cdot (-y)^p,
\end{equation}
and similarly, the pushforward $a_{X*}IT_{y*}(X)$ is the {\it intersection cohomology Hodge polynomial}
\begin{equation}\label{defichi}I\chi_y(X):=\chi_y([IH^*(X;\bQ)])=\sum_{i,p}
(-1)^i \dim_{\bC} \Gr^p_F IH^i(X;\bC) \cdot (-y)^p
\end{equation}
$$=\sum_{p,q}
(-1)^q Ih^{p,q}(X) \cdot y^p,$$
where the last equality uses the purity of the Hodge structure on the intersection cohomology groups $$IH^*(X;\bQ)=H^*(X;IC'^H_X),$$ and 
$$Ih^{p,q}(X):=\dim_{\bC} \Gr^p_F IH^{p+q}(X;\bC)$$
are the corresponding intersection cohomology Hodge numbers. Both $\chi_y(X)$ and $I\chi_y(X)$ are natural generalizations to the singular context of the Hirzebruch $\chi_y$-genus of \eqref{chiy} from the smooth compact case.\\

A Hodge index theorem for singular varieties has been obtained in \cite{MSS0} in the projective case, and this was extended to all compact complex algebraic varieties of pure dimension in \cite[formula (6)]{FPS}. It can be formulated as the equality
\begin{equation}\label{ihif}
\sigma(X)=I\chi_1(X),
\end{equation}
where the signature $\sigma(X)$ of $X$ is the Gorresky-MacPherson signature \cite{GM} defined via the Poincar\'e duality pairing on the middle intersection cohomology group of $X$.

In Section \ref{Lclass} below we recall the definition of the homology $L$-classes for singular compact complex pure-dimensional varieties, which at the degree level yields the intersection cohomology signature. Given that $I\chi_1(X)$ is the degree of the class $IT_{1*}(X)$, it is then natural to expect that the following holds, see \cite[Remark 5.4]{BSY}.
\begin{conjecture}[Characteristic class version of Hodge index theorem] \label{BSY}
Let $X$ be a compact pure-dimensional complex algebraic variety. Then
\begin{equation}\label{conj}
L_*(X)=IT_{1*}(X) \in H_{2*}(X;\bQ).
\end{equation}
\end{conjecture}
For rational homology manifolds, the above conjecture has been recently proved in \cite{FPS}, see also \cite{FP} for the projective case. More generally, as explained in Remark \ref{r11}, the conjecture holds whenever $[IC_X'^H]\in \im(\chi_{Hdg}).$

\begin{remark}
    By the multiplicativity of both sides of the equality \eqref{conj} (cf. \cite[Proposition 5.16]{W} for $L$-classes of Witt spaces and, resp., \cite[Section 5]{S} for $IT_{1*}$), if Conjecture \ref{BSY} holds for both $X$ and $Y$, then it holds for their product $X \times Y$.
\end{remark}

\begin{remark}
Note that for $X$ a compact variety of pure dimension $d_X$, $L_k(X)=0\in H_{2k}(X;\bQ)$ by definition if $d_X+k$ is odd (as explained in Section \ref{Lclass} below). Similarly,  $IT_{1,k}(X)=0\in H^{BM}_{2k}(X;\bQ)$ if $d_X+k$ is odd, by a duality argument (see \cite[Example 2.11]{MS2}). Here the compactness is only needed in the definition of $L$-classes, but not for the last statement about $IT_{1*}(X)$.
\end{remark}


\subsection{Cobordism groups and relation to mixed Hodge modules}\label{cobg}
In this section we recall two notions of a cobordism group of self-dual constructible complexes, one defined in terms of Verdier duality \cite{CS,Ban0,Yo,BSY} appearing in the topological definition of $L$-classes, and the other one, coming up in relation to mixed Hodge modules, defined in terms of non-degenerate pairings, cf. \cite{FPS,FPS2}.


Let $X$ be a complex algebraic variety and let $\bK$ be a subfield of $\bR$. A self-dual complex $(\cF, \alpha)$ on $X$ consists of a bounded $\bK$-complex $\cF \in D^b_c(X;\bK)$ with constructible cohomology, which is endowed with a 
self-duality isomorphism 
$$\alpha:\cF \to \cD_X(\cF)=Rhom(\cF,\bD_X),$$ where $\cD_X(-)$ is the Verdier duality functor on $X$, and $\bD_X=a_X^!\bK_{pt}$ is the $\bK$-dualizing complex on $X$, with $a_X:X \to pt$ the constant map.
In the following we implicitly assume that $\bD_X$ is represented by a bounded below complex of injective sheaves.

A pairing $S:\cF \otimes \cF \to \bD_X$ corresponds by the adjunction formula
\begin{equation}\label{adj} \theta: Hom_{D^b_c(X;\bK)}(\cF\otimes \cG, \bD_X) = 
Hom_{D^b_c(X;\bK)}(\cF, \cD_X(\cG))\end{equation}
to a morphism  $\alpha_S:\cF \to \cD_X(\cF)$. 
 Such a pairing $S$ is called {\it perfect} if  $\alpha_S$ is an isomorphism. The pairing $S$ is called {\it symmetric} (or {\it skew-symmetric}) if $S \circ \iota=S$ (or $-S$), where $$\iota :\cF \otimes \cG \to \cG \otimes \cF$$ is the standard involution on the tensor product of complexes (e.g., see \cite[1.3.4]{MSS0}). In terms of self-dual complexes, the notion of (skew-)symmetry can be defined by using the canonical biduality isomorphism $can:\cF \to \cD_X(\cD_X(\cF))$ via
 \[ \cD_X(\alpha)\circ can=\pm \alpha,\]
 see, e.g., \cite[Section 4]{BSY} or \cite{Yo}.
 Here, the biduality isomorphism $can$ is usually part of the data needed to define a cobordism group of self-dual objects, e.g., switching from $can$ to $-can$ switches the role of (skew-)symmetric objects. But in our context there is indeed a {\it canonical} choice so that $can$ corresponds to the pairing
 $$\cF \otimes Rhom(\cF,\bD_X) \simeq Rhom(\cF,\bD_X) \otimes \cF \to \bD_X \:,$$
 with the last pairing corresponding via the adjunction formula \eqref{adj} to the identity map of $Rhom(\cF,\bD_X)$ (see, e.g., \cite[Section 3.2]{CH}).
If $\alpha:\cF \to \cD_X(\cG)$ corresponds to the pairing 
$\theta(\alpha): \cF  \otimes\cG \to \bD_X $, then $\theta(\cD_X(\alpha)\circ can)$ corresponds to the pairing
$$\theta(\alpha) \circ \iota :\quad  \cG  \otimes\cF \simeq   \cF  \otimes\cG \to \bD_X \:. $$
 
 For instance, if $X$ is smooth and pure dimensional with $d_X=\dim(X)$, then $\bD_X\cong \bK_X[2d_X]$, so the canonical self-pairing $S_{can}$ on $\bK_X[d_X]$ is $(-1)^{d_X}$-symmetric (i.e., symmetric if $d_X$ is even, and skew-symmetric if $d_X$ is odd).
This corresponds to the canonical self-duality isomorphism $\alpha_{can}$ of $\bK_X[d_X]$.

The (skew-)symmetric {\it cobordism group} $\Omega_{\bK\pm}(X)
:=\Omega_{\pm}(D^b_c(X;\bK),\bD_X, can)$ is defined as the quotient of the monoid of isomorphism classes of (skew-)symmetric self-dual complexes, by a certain cobordism relation defined in terms of ($\pm$)self-dual octahedral diagrams as in \cite{Yo, BSY}. Set
\[\Omega_\bK(X):=\Omega_{\bK+}(X) \oplus \Omega_{\bK-}(X).\]

Note that these are functorial for proper morphisms, since $Rf_*$ commutes with Verdier duality for a proper morphism $f$.

\begin{remark}\label{sign}
    Let us emphasize that in the description of elements in the cobordism groups (defined either by duality or pairings), as well as in the definition of $L$-classes via signatures, sign conventions are important. E.g., a sign isomorphism can introduce negative elements or $L$-classes. Here we have to specify which of the two possible sign conventions we are using for 
    \begin{itemize}
        \item[(i)] the tensor product, as used in the context of pairings;
        \item[(ii)] the identification in \eqref{adj};
        \item[(iii)] the duality functor $\cD_X(-)=Rhom(-,\bD_X)$.
    \end{itemize}
    Since we want to compare with Saito's theory of algebraic mixed Hodge modules and the construction of the transformation $\Pol^{pair}$ from \cite{FPS}, we follow here exactly the sign conventions from \cite{Saito88, Saito89} (compare also with \cite[Appendix A]{CH}), namely:
    \begin{itemize}
        \item[(i)] For the tensor product, we use the convention
        \[ R[n] \otimes_R (-)=(-)[n] \] without any signs, for a coefficient ring like $R=\bZ, \bQ, \bR$. So moving out a shift in the first variable of the tensor product does not produce a sign, whereas a shift in the second factor could introduce a sign dictated by the symmetry isomorphism $\iota$ given by multiplication by $(-1)^{ij}$ on $\cF^i \otimes \cG^j$. More precisely,
        \[
        \cF^i \otimes_R (R[n]) \otimes_R \cG^j) = 
        (\cF^i \otimes_R R[n]) \otimes_R \cG^j
         \simeq 
        (R[n] \otimes_R \cF^i) \otimes_R \cG^j
        =(\cF^i \otimes_R \cG^j)[n],
        \]
        where the associativity is an equality without signs, and the middle symmetry isomorphism produces the sign $(-1)^{in}$.
        \item[(ii)] the identification in \eqref{adj} is an equality without any signs, as in  \cite[(1.8.1)]{Saito89} and \cite{FPS}.
        \item[(iii)] For $Rhom$ we choose the convention from \cite[(1.8.2)]{Saito89}, so that moving out a shift in the second variable does not produce any sign (see \cite[(1.8.4)]{Saito89}), whereas the sign $(-1)^{in}$ appears in the isomorphism
        \begin{equation}\label{th1}
        th_n: \:hom(\cF^i[-n],\cG^{i+j}) \simeq hom(\cF^i,\cG^{i+j})[n]\:,
        \end{equation}
        see \cite[(1.8.3)]{Saito89}. This means that the Verdier duality functor $\cD_X$ is not {\it strict} in the sense of \cite[Remark 2.1.2]{CH}.
        \end{itemize}
        Note that an even shift does not involve any signs. Also, these sign conventions fit with those from \cite{Ban0}. The composition of the sign-isomorphism $th_n$ for different shifts $[n]$ only commutes  up to the following sign:
        \begin{equation}\label{th2}\begin{CD}
         hom(\cF^i[-(n+m)],\cG^{i+j})  @> th_n > \sim >
         hom(\cF^i[-m],\cG^{i+j})[n] \\
        @V  th_{n+m} V \wr V @V \wr V th_m V\\
         hom(\cF^i,\cG^{i+j})[n+m] 
         @> (-1)^{nm}\cdot id > \sim >
         hom(\cF^i,\cG^{i+j})[n+m] \:,
        \end{CD}\end{equation}
        i.e., only for $n$ and $m$ both odd an additional minus sign shows up. For example the inverse of 
        $th_n $ is $(-1)^n\cdot th_{-n}$.
\end{remark}

Let us  mention here that a self-dual complex $(\cF, \alpha)$ is cobordant to its zeroth perverse cohomology ${^p\cH}^0(\cF,\alpha)$ (cf. \cite[Example 6.6]{Yo}), hence they give the same element 
\begin{equation}\label{cob}
 [(\cF, \alpha)]=   [{^p\cH}^0(\cF,\alpha)] \in \Omega_\bK(X).
\end{equation} 
Here we use the middle perversity, which is self-dual with respect to Verdier duality, so that we get an induced duality isomorphism:
$${^p\cH}^0(\alpha):{^p\cH}^0(\cF) \to {^p\cH}^0(\cD_X(\cF))=\cD_X({^p\cH}^0(\cF)).$$

Self-duality isomorphisms and cobordism relations are stable by the derived pushforward $Rf_*$ under a proper morphism of complex varieties $f:X \to Y$, so one gets a pushforward 
\[f_*:\Omega_\bK(X) \to \Omega_\bK(Y)\]
compatible with the above decomposition.\\

Following \cite[Section 4]{BSY}, one has a natural transformation
\[ \sd_\bK:K_0(var/X) \to \Omega_\bK(X), \]
which on a generator $[f:Z\to X]$ with $Z$ smooth pure dimensional and proper
(or projective) over $X$ is defined by $\sd_\bK([f])=f_*[(\bK_Z[d_Z], \alpha_{can})]$, with $d_Z=\dim(Z)$. Note that the construction of $\sd_\bK$ in \cite{BSY} uses Bittner's presentation \cite{Bi} of $K_0(var/X)$ in terms of a blow-up relation coming from the weak factorization theorem.
It is important to also note that in general  $\sd_\bK([id_X])\neq [IC_X]$.

\begin{remark}
    The cobordism class $[(\bK_Z[d_Z], \alpha_{can})]$ agrees with the corresponding class originally introduced in \cite[Theorem 4.2]{BSY}
    via the multiplication map ${\rm m}: \bK_Z\otimes \bK_Z\to \bK_Z$. Indeed, for $Z$ smooth of pure dimension $d_Z$, \cite[Example 4.1]{BSY}
    starts  more generally with a (non-degenerate) pairing 
    $$S: \bL \otimes \bL \to \bK_Z$$ of a local system $\bL$ on $Z$ in degree zero, corresponding via the adjunction \eqref{adj} to an isomorphism $$\alpha: \bL\to \bL^{\vee} =hom(\bL,\bK_Z)=Rhom(\bL,\bK_Z)\:,$$
     which is then shifted by $[m+d_Z]$ (for $m\in \bZ$ fixed,
     with $m=0$ corresponding to  \cite[Example 4.1]{BSY}):
     \begin{equation}\begin{split}
         \alpha[m+d_Z]: \bL[m+d_Z] &\to 
         Rhom(\bL,\bK_Z)[-(m+d_Z)][2(m+d_Z)]\\
     &\simeq Rhom(\bL[m+d_Z],\bK_Z[2(m+d_Z)]) \\ &=\cD_X(\bL[m+d_Z])[2m]
     \end{split} \end{equation}
     (this doesn't introduce any signs since $\bL$ is in degree zero). Then $\alpha$ can be recovered via the induced map
     $$\bL=\cH^{-(m+d_Z)}(\bL[m+d_Z]) \to \cH^{-(m+d_Z)}(\cD_X(\bL[m+d_Z])[2m])=\bL^{\vee}\:. $$
     Similarly $\alpha[d_Z]$ corresponds (for $m=0$) via \eqref{adj} to a shifted pairing
     \begin{equation}
      S' : \bL[d_Z] \otimes \bL[d_Z] \to \bD_Z=\bK[2d_Z] \:,
     \end{equation}
     giving back $S$ via
     $$\bL\otimes \bL = \cH^{-2d_Z}(\bL[d_Z]\otimes \bL[d_Z]) \to \cH^{-2d_Z}(\bD_Z)=\bK_Z\:.$$
\end{remark}

The following result was proved in \cite[Proposition 1, Remark 1.2c]{FPS}.
\begin{proposition}\cite{FPS}\label{pr1} \ Let $M$ be a pure ($\bQ$-)Hodge module of weight $w$ on $X$, and let $\cF_\bR:=\cF_\bK \otimes_\bK \bR$ be its underlying $\bR$-complex. Let 
\[S_\bR: \cF_\bR \otimes \cF_\bR \to \bD_X(-w) \simeq \bD_X\]
be the scalar extension of a polarization $S$ on the ($\bQ$-)Hodge module $M$, and let $\alpha_\bR$ be the corresponding self-duality isomorphism. Then the cobordism class $[(\cF_\bR, \alpha_\bR)] \in \Omega_\bR(X)$ does not depend on the choice of a polarization $S$.
\end{proposition}

\begin{remark}
The isomorphism $\bD_X(-w) \simeq \bD_X$ comes from ommitting in this topological context the Tate twists via the choice of $i=\sqrt{-1}\in \bC$
coming from the complex orientation of $\bC$.
    The reader should be aware that the original result of \cite[Proposition 1]{FPS} is formulated in terms of an a priori different notion of cobordism group $\Omega^{pair}_\bR(X)$ defined in terms of perfect pairings. However, by \cite[Remark 1.2c]{FPS}, even the isomorphism class $[(\cF_\bR, \alpha_\bR)] \in \Omega_\bR(X)$ does not depend on the choice of a polarization $S$. For this, it is important to work (with sheaves) over the real numbers, where square roots of positive real numbers are available (see \cite[Section 2.1]{FPS} for more details).
\end{remark}


From this one can deduce the following important result,
based on M. Saito's stability result of pure polarized Hodge modules for projective morphisms, including the relative hard 
Lefschetz theorem and the corresponding relative Hodge-Riemann bilinear relations. 
\begin{thm}\cite[Theorem 3]{FPS}\label{thpol} \
There is a natural transformation 
\[ \Pol:K_0(\MHM(X)) \to \Omega_\bR(X)\]
defined by \begin{equation}\label{pol} \Pol([M])=[(\cF_\bR, (-1)^{w(w+1)/2}\alpha_\bR)] \in \Omega_\bR(X),\end{equation}
for a pure ($\bQ$-)Hodge module $M$ of weight $w$ on $X$, with $(\cF_\bR, \alpha_\bR)$ as in Proposition \ref{pr1}. This transformation is compatible with the pushforward by a proper morphism $f:X \to Y$, i.e., \[\Pol \circ f_*=f_* \circ \Pol.\]
\end{thm}

\begin{remark}
As already pointed out before, Theorem \ref{thpol} is originally stated in \cite{FPS} in terms of the cobordism group $\Omega^{pair}_\bR(X)$, but its proof only uses Proposition \ref{pr1} together with \eqref{cob} (the counterpart of which consists of \cite[Proposition 1.1a]{FPS}), so the proof also applies to the cobordism group $\Omega_\bR(X)$. 
\end{remark}

This implies the commutativity of the following part of diagram \eqref{diag1}, i.e., it gives the intrinsic meaning of the transformation 
$\sd_\bR$.
\begin{corollary}\cite[Corollary 1]{FPS}
One has a commutative diagram
\begin{equation}
\begin{tikzcd}\label{diag2}
K_0(var/X) \arrow{r}{\chi_{Hdg}}\arrow{rd}[swap]{\sd_\bR}
&K_0(\MHM(X)) \arrow{d}{\Pol}\\
&\Omega_\bR(X)\:.
\end{tikzcd}
\end{equation}
\end{corollary}

\begin{proof}
We include the proof here, in order to get a better understanding of the signs involved in the construction of $\Pol$. Since $K_0(var/X)$ is generated by classes of proper maps $Y \to X$ with $Y$ smooth, and all morphisms in (\ref{diag2}) commute with proper pushforwards, it suffices to verify the commutativity of (\ref{diag2}) for $[id_X]$ with $X$  smooth.
For simplicity, we assume $\bK=\bQ$.

Let $d_X=\dim (X)$ and denote by $\bQ_X$ the constant variation of weight zero on $X$ with the trivial pairing
 \[S: \bQ_X \otimes\bQ_X \rightarrow \bQ_X\] 
 given by multiplication. 
 Consider also the induced pairing
 \[ S_{can}: \bQ_X[d_X] \otimes \bQ_X[d_X] \rightarrow \bQ_X[2d_X],\]
 given by
 \[S: \bQ_X \otimes\bQ_X = \cH^{-2d_X}(\bQ_X[d_X] \otimes \bQ_X[d_X])\rightarrow \bQ_X .\]
Then,
 by definition, 
 \[ \sd_{\bR}([id_X])=[(\bR_X[d_X], \alpha_{can})]\in \Omega_\bR(X).\]
On the other hand, 
 \[
  \chi_{Hdg}([id_X])=\bQ_X^H=\bQ_X^H[d_X][-d_X] \in K_0(\MHM(X)), 
  \]  
where $\bQ_X^H[d_X]$ is the pure ($\bQ$-)Hodge module of weight $d_X$ on $X$, with polarization
 \[S^H: \bQ_X[d_X] \otimes \bQ_X[d_X] \rightarrow \bD_X(-d_X)\simeq \bQ_X[2d_X].\]
 By \cite[Lemma 5.2.12 and 5.4.1]{Saito88}, one has the identification \begin{equation}\label{Sai} S^H= (-1)^{d_X(d_X-1)/2}S_{can}.\end{equation}

Finally, using the definition of $\Pol$ from \eqref{pol}, together with \eqref{Sai}  and the identification 
$$[\bQ^H_X]=(-1)^{d_X}[\bQ^H_X[d_X]] \in K_0(\MHM(X)),$$
we have
\begin{equation*}
\begin{split}
\Pol([\bQ_X^H])&=(-1)^{-d_X}\Pol([\bQ_X^H[d_X]])=
(-1)^{-d_X}(-1)^{d_X(d_X+1)/2}[\bR_X[d_X], (\alpha_{S^H})_{\bR}]\\&=(-1)^{d_X(d_X-1)/2}[\bR_X[d_X], (\alpha_{S^H})_{\bR}]=
    [\bR_X[d_X],  \alpha_{can}]=\sd_\bR([id_X])
    \end{split}
\end{equation*}
\end{proof}

\begin{remark}
    Note that only our formulation of the above result in terms of the cobordism group $\Omega_\bR(X)$ fits with the original definition of $\sd_\bR$ in \cite{BSY}, whereas \cite{FPS} defines a corresponding transformation $$\sd^{pair}_\bR:K_0(var/X)\to \Omega^{pair}_\bR(X)$$
    so that $\tau \circ \sd_\bR=\sd_\bR^{pair}.$
   Here the canonical surjection $\tau:\Omega_\bR(X) \to \Omega^{pair}_\bR(X)$ comes from \cite[Remark 1.1d]{FPS}, together with the fact that we follow exactly the sign conventions of \cite{FPS}. Moreover, as noted in \cite[Remark 1.2c]{FPS}, this group homomorphism of cobordism groups is even an isomorphism.
\end{remark}

\begin{remark}\label{ext-pol}
    The above argument in the smooth case can be extended to show that for $X$ a complex algebraic variety of pure dimension $d_X$ one has
    \[\Pol([IC'^H_X])=[IC_X, \alpha_{can}] \in \Omega_\bR(X)\]
    with the natural self-duality (or pairing) on $IC_X$ extended uniquely from that of the smooth locus. Similarly, M. Saito's work allows one to extend the polarization $S$ of the constant variation $\bQ_{X_{reg}}$ on the smooth locus $X_{reg}$ of $X$ to a polarization $S^H$  of the pure Hodge module $IC_X^H$ with the same sign as in \eqref{Sai}, see also \cite{FPS} for more details. 
\end{remark}

The commutativity of diagram \eqref{diag2} is then used to prove the following theorem (see \cite[Theorem 1]{FPS} for a more general formulation):
\begin{thm}\label{pfs}
If $X$ is a pure dimensional  complex algebraic variety which is a rational homology manifold, then 
\begin{equation}\label{eqmain}
\sd_\bR([id_X])=[IC_X,\alpha_{can}] \in \Omega_\bR(X).
\end{equation}
\end{thm}

\begin{remark}
    The sign $(-1)^{w(w+1)/2}$ in the definition \eqref{pol} of the transformation $\Pol$ is designed for the proof of functoriality given in \cite{FPS}, as well as the proof of \eqref{eqmain}.
\end{remark}

Note that $\bQ^H_X\simeq IC'^H_X$ for $X$ a rational homology manifold.
As noted in \cite[Proposition 2 and Section 2.5]{FPS}, formula \eqref{eqmain} does {\it not} hold for $\bK=\bQ$, the example given in loc.cit. being that of a compact surface with only one $A_1$ (ordinary double point) singularity. This is due to the richer structure  of the Witt group of symmetric bilinear forms over rationals, compared to that over the reals.


\subsection{L-classes}\label{Lclass}
Stratified bordism invariant $L$-classes $L_*(X)$ for compact oriented stratified pseudo-manifolds $X$ with only even (co)dimension strata have been introduced in \cite{GM} by Goresky and MacPherson, relying on a classical Pontrjagin construction. More generally, Cappell-Shaneson \cite{CS} (see also \cite[Section 8.2.3]{Ban0}) introduced in this context an $L$-class $L_*^{CS}$ for elements in a suitable cobordism group 
$$\Omega^{CS}_{\cS, \bK}(X,\cD_X[2\dim(X)]):=\Omega(D^b_{\cS-c}(X;\bK),\cD_X[2\dim(X)], can) $$ 
of  $\bK$-constructible complexes $\cF$ which are self-dual with respect to a shifted Verdier duality (and working  with a fixed Whitney stratification $\cS$  with only even dimensional strata). 
But as we now explain, to construct a functorial $L$-class transformation $L_*$  via the methods of Cappell-Shaneson \cite{CS} it is more natural to work with the (unshifted) Verdier duality, using the 
limit with respect to refinements of algebraic Whitney stratifications $\cS$ of $X$, 
 so that one gets an $L$-class transformation: $$L_*: \Omega_\bK(X):=\Omega_\bK(X, \cD_X) \simeq  \lim_{\cS}\:\Omega^{CS}_{\cS,\bK}(X,\cD_X) \to H_*(X) \otimes \bQ=H_{2*}(X;\bQ),$$ which will be functorial for  pushforwards under morphisms of compact varieties. In this setting, the Goresky-MacPherson $L$-class will be given by $$L_*(X)= L_*([IC_X,\alpha_{can}]).$$
(Here we can work as before with a subfield $\bK$ of $\bR$, but the $L$-class transformation only depends on the associated complexes $\cF_\bR$ with real coefficients, as used before.) 

\begin{remark}
    From the viewpoint of algebraic geometry, and in particular for the use of mixed Hodge modules, it is more natural to work with the cobordism group $\Omega_\bK(X):=\Omega_\bK(X, \cD_X)$ defined via Verdier duality without a shift. On the other hand, in topology, especially in relation to intersection cohomology, most often different conventions are used. This is in particular the case for the $L$-classes $L_*^{CS} $ of Cappell-Shaneson \cite{CS} defined on $\Omega^{CS}_{\cS,\bK}(X,\cD_X[2\dim(X)])$ and taking $\dim(X)$ into account. 
    While Cappell-Shaneson do not specify in \cite{CS} the sign convention for the $Rhom$ functor, the $L$-class transformation evaluated on (shifted) intersection complexes should be the Goresky-MacPherson $L$-class. As we will see later, this forces the choice of the sign convention as in Remark \ref{sign}.
\end{remark}

Let us now explain the construction of this functorial $L$-class, directly in terms of our cobordism group $\Omega_\bK(X)=\Omega_\bK(X, \cD_X)$, i.e., we work with the perverse conventions. Let $X$ be a compact  complex algebraic variety of pure dimension $d_X$, together with a self-dual complex $(\cF, \alpha)$, with $\cF\in D^b_c(X;\bK)$ and $\alpha$ the corresponding self-duality isomorphism. 
Fix a complex algebraic Whitney stratification $\cS$ adapted to $\cF$ (i.e., the cohomology sheaves of $\cF$ are locally constant on the strata of the stratification $\cS$), so that $\cF$ belongs to the corresponding triangulated
subcategory $D^b_{\cS-c}(X;\bK)$ of $\cS$-constructible complexes.
This subcategory is also stable under  Verdier duality and fits with the perverse t-structure.

\begin{remark} To make better visible a comparison between the different 
conventions mentioned above, we work from the beginning with the cobordism group
$$\Omega^{CS}_{\cS, \bK}(X,\cD_X[2m]):=\Omega(D^b_{\cS-c}(X;\bK),\cD_X[2m], can) $$ 
with respect to a {\it shifted Verdier duality} (for $m\in \bZ$):
$$\cD_X[2m](\cF):= (\cD_X(\cF))[2m]=Rhom(\cF,\bD_X)[2m]
=Rhom(\cF,\bD_X[2m]) \:.$$
Here the (shifted) canonical biduality isomorphism $can$ 
for $\cF \in D^b_{\cS-c}(X;\bK)$ is defined via
$$can: \cF\to \cD_X(\cD_X(\cF)) = \cD_X\left(\cD_X(\cF)[2m]\right)[2m]= \cD_X[2m](\cD_X[2m](\cF))\:.
$$
The case $m=\dim(X)$ fits with the original conventions of Cappell-Shaneson \cite{CS} (see also \cite[Section 8.2.3]{Ban0}), and the case $m=-\dim(X)$ fits with the conventions of \cite{W},
whereas our convention of the use of the unshifted Verdier duality corresponds to the case $m=0$. 
Note that the shift functor $[m]$ induces for $n\in \bZ$ an isomorphism
(see also \cite[Section 4]{BSY})
\begin{equation}\begin{CD}
    \Omega^{CS}_{\cS, \bK}(X,\cD_X[2n])
  @ > [m] > \sim > \Omega^{CS}_{\cS, \bK}(X,\cD_X[2(n+m)])\:,
\end{CD}\end{equation}
which associates to a class $[\cF,\alpha]$ the class of
$$\alpha[m]: \cF[m]\to \cD_X(\cF)[2n][m]=\cD_X(\cF)[-m][2(n+m)]
\simeq \cD_X[2(n+m)](\cF[m])\:.$$
This commutes with both dualities and the biduality isomorphism $can$.
For $m$ even it is a triangulated functor (i.e., mapping distinguished triangles to distinguished triangles), 
so that it maps a {\it self-dual octahedral diagram} 
(used for the cobordism relation as in \cite[Section 4]{BSY} 
and \cite[Section 6]{Yo}) to itself
preserving the (skew-)symmetry. But for $m$ odd, it is only a $-1$-exact functor of triangulated categories, i.e., maps distinguished triangles to distinguished triangles up to an additional minus sign in the degree one morphism of the triangle. 
So it also maps a {\it self-dual octahedral diagram}
to itself, if one adds a minus sign for all degree one maps in the resulting diagram, also preserving the (skew-)symmetry.
This is enough to be compatible with the cobordism relation.
\end{remark}

Let $[\cF,\alpha]\in \Omega^{CS}_{\cS, \bK}(X,\cD_X)$ be a given cobordism class.
For each non-negative integer $k \in \bZ$, one associates a homology class $$L_k([\cF,\alpha])\in H_{2k}(X;\bQ)$$ as follows. 
Let $f:X \to S^{2k}$ be a smooth map (i.e., locally given by a restriction to $X$ of a smooth function on a local model). Assume moreover that $f$ is transverse to the stratification $\cS$ (in local models), in the sense that if $N$ is the North pole then $X_f:=f^{-1}(N)$ is transverse to the chosen Whitney stratification $\cS$ of $X$ adapted to $\cF$ (i.e., $N$ is not a stratified critical value of $f$). Then $i_f:X_f \hookrightarrow X$ is a normally nonsingular embedding with trivial normal bundle, so $i_f^!=i_f^*[-2k]$, when evaluated on 
$\cF\in D^b_{\cS-c}(X;\bK)$. If follows that 
\[\cF_f:=i_f^!\cF[k]\simeq i_f^*\cF[-k]\] is self-dual on $X_f$ with respect to Verdier duality for an induced self-duality isomorphism $\alpha_f$, depending on a choice of orientation of the normal bundle given by an orientation of the sphere $S^{2k}$. More precisely, if we consider 
\begin{equation}\label{indu}i_f^!\alpha:\:
i_f^!\cF \to  i_f^!Rhom(\cF,\bD_X)=Rhom(i_f^*\cF, i_f^!\bD_X)=Rhom(i_f^!\cF, \bD_{X_f})[-2k],\end{equation}
then $\alpha_f$ is defined as \begin{equation}\label{ind}\alpha_f:=i_f^!\alpha[k]:i_f^*\cF[-k] \to Rhom(i_f^*\cF, \bD_{X_f})[k]\simeq Rhom(i_f^*\cF[-k], \bD_{X_f}).\end{equation}
Here the shift $[k]$ doesn't introduce signs, but the last isomorphism induces by \eqref{th1} the sign $(-1)^{ik}$ on
$$hom((i_f^*\cF)^i,(\bD_{X_f})^{i+j})\:.$$
Note that $\alpha_f$ has the same (resp., opposite) {\it symmetry property} as $\alpha$ if $k$ is even (resp., odd). 
The reason is that the isomorphism $i_f^!\cF[k]\simeq i_f^*\cF[-k]$ is compatible with the biduality morphism $can$ only up to the sign $(-1)^k$.
In fact, $i_f$ can be locally viewed as the zero-section $s$ of an open tubular neighborhood $U$ of $X_f$ in $X$, with smooth oriented projection $p: U\to X_f$ of real fiber dimension $2k$ (and $p\circ s =id$) so that the transformation $p^![-k]\simeq p^*[k]$ locally corresponds to $- \boxtimes \bK_{\bC^k}[k]$,
with $(\bK_{\bC^k}[k],\alpha_{can})$ only $(-1)^k$-symmetric.\\

So one gets an induced self-duality isomorphism, hence also a nondegenerate pairing on 
$$H^{k}(X_f;i_f^!\cF)\otimes_\bK \bR\:,$$  whose signature we denote by $\sigma_f([\cF,\alpha]))$; this is set to equal $0$ in the case when $i_f^!\alpha$ is skew-symmetric (i.e., when $k$ is odd for $\alpha$ symmetric, or $k$ is even for $\alpha$ skew-symmetric). Note that by our sign conventions \eqref{th1} for the isomorphism $th_k$ and the used shift $[k]$, this only agrees up to the sign $(-1)^{k\cdot k}=(-1)^k$ with the signature of the  pairing induced via $\alpha_f=i_f^!\alpha[k]$ on
$$H^{0}(X_f;\cF_f)\otimes_\bK \bR\:.$$

The assignment
\[\sigma:\pi^{2k}(X) \to \bZ, \ \ [f] \mapsto \sigma_f([\cF,\alpha])\]
is for $4k-1>2d_X$ a well-defined homomorphism on the $2k$-cohomotopy group of $X$ (via approximation by smooth maps, stratification refinements, and using the {\it bordism invariance} of this signature as in \cite[Prop. 8.2.8]{Ban0}. Note that $\pi^{2k}(X)$ has an abelian group structure only for 
$4k-1>2d_X$). Next, recall that, by Serre's theorem, the Hurewicz map $\pi^{2k}(X) \otimes \bQ \to H^{2k}(X;\bQ)$, for the same choice of orientation of the sphere $S^{2k}$, is also an isomorphism for $4k-1>2d_X$, so in this case the above homomorphism is a linear functional \[ \sigma \otimes \bQ \in \Hom(H^{2k}(X;\bQ), \bQ)\cong H_{2k}(X;\bQ).\] Define
\[L_k([\cF,\alpha]):=\sigma\otimes \bQ \in H_{2k}(X;\bQ).\]
This is independent on the choice of orientation of $S^{2k}$. 
In order to remove the dimensional restriction, one may replace $X$ by a product $X \times S^{2r}$ for $r$ large enough, and $\cF$ by a (shifted) pullback to the product, as explained in \cite{CS,Ban0}. (Even though $X \times S^{2r}$ is not a complex variety, it gets an induced oriented Whitney stratification  with only even dimension strata, adapted to the pullback of $\cF$.) The above construction is compatible with refinements of such algebraic Whitney stratifications and yields a natural transformation
\[L_*:\Omega_\bK(X) \simeq  \lim_{\cS}\:\Omega^{CS}_{\cS,\bK}(X,\cD_X) 
\to H_*(X) \otimes \bQ=H_{2*}(X;\bQ),\]
which is compatible with pushforward by a morphism $g: X\to X'$ of compact complex algebraic varieties (by using a proper base change together with suitable complex algebraic Whitney stratifications of $X$ and $X'$ in such a way that $g$ is a proper stratified submersion). This follows from the following facts: 
\begin{enumerate}
\item[i)] the used transformations $i_f^![k]\simeq i_f^*[-k]$ (or $i_f^!$) and $Rg_*$ for a proper morphism $g$, commute (up to a shift)  with Verdier duality (in the sense of \cite[Definition 2.2.1]{CH}), hence they induce transformations at the (shifted) cobordism group level:
\begin{equation*}\begin{CD}
  \Omega^{CS}_{\cS,\bK}(X, \cD_X) @> i_f^! >>  
  \Omega^{CS}_{X_f\cap \cS,\bK}(X_f, \cD_{X_f}[-2k]) 
  @> [k] >> \Omega^{CS}_{X_f\cap \cS,\bK}(X_f, \cD_{X_f})\:.
\end{CD}\end{equation*}
\item[ii)] the signature $\sigma:\Omega_\bR(pt) \simeq Witt(\bR) \to \bZ$ factorizes over this cobordism group, which in this case (using the standard t-structure, or a shift of this) is isomorphic to the usual Witt groups of real vector spaces with (skew-)symmetric pairings (see \cite[Theorem 7.4]{Yo}).
In this case, the signature is even a ring isomorphism.
\end{enumerate}

\begin{remark}
 Note that our definition of $\sigma_f([\cF,\alpha])$ only agrees up to a sign with the original  definition of Cappell-Shaneson \cite{CS} (see also \cite[Section 8.2.3]{Ban0}), i.e, the following diagram commutes:
 \begin{equation*}\begin{CD}
   \Omega^{CS}_{\cS,\bK}(X, \cD_X) @> i_f^! >>  
  \Omega^{CS}_{X_f\cap \cS,\bK}(X_f, \cD_{X_f}[-2k]) 
  @> \sigma(H^{k}(X_f,i_f^! -)) >>   \bZ\\
  @V [d_X] VV @V [d_X] VV @VV (-1)^{d_X\cdot k} \cdot id V\\
 \Omega^{CS}_{\cS,\bK}(X, \cD_X[2d_X]) @> i_f^! >>  
  \Omega^{CS}_{X_f\cap \cS,\bK}(X_f, \cD_{X_f}[2(d_X-k)]) 
  @> \sigma(H^{-(d_X-k)}(X_f,i_f^! -)) >>   \bZ\:. 
\end{CD} \end{equation*}
So the corresponding $L$-classes agree for symmetric self-dual complexes, and differ by a sign $(-1)^{d_X}$ for skew-symmetric self-dual complexes.
\end{remark}


\begin{example}\label{ex-twIC1} Let $\cF=IC_X(\cL)$ be a twisted intersection cohomology complex on a  variety $X$ of pure dimension $d_X=\dim(X)$, with $\cL$ a self-dual local system with duality $\alpha: \cL\to \cL^{\vee}$ on an open Zariski-dense smooth subset $U$ of $X$ (given as the top-dimensional strata of an algebraic Whitney stratification of $X$). Then 
$$i_f^*(IC_X(\cL))[-k] = IC_{X_f}(\cL|_{U_f})$$
is also a twisted intersection cohomology complex, where $U_f=U \cap X_f$.
Similarly for the shifted  
twisted intersection cohomology complexes
$$i_f^!(IC_X(\cL)) = IC_{X_f}(\cL|_{U_f})[-k]
\quad \text{and} \quad i_f^!(IC_X(\cL)[d_X]) = IC_{X_f}(\cL|_{U_f})[d_X-k] \:.$$
And the corresponding induced self-duality is uniquely fixed by that of the corresponding shifted local systems,
where the following sign changes show up:
\begin{enumerate}
    \item[(a)] For $i_f^*(\cL[d_X])[-k] = \cL|_{U_f}[d_X-k]$
  the self-duality isomorphism  is multiplied by \eqref{th2} with
the sign $(-1)^{d_X\cdot k}$. In addition the factor $(-1)^{k\cdot k}$ shows up for the calculation of the correct signature used for the definition of our $L$-classes.
Altogether the self-duality isomorphism  is multiplied with
the sign $(-1)^{(d_X-k)\cdot k}$.
\item[(b)] The restriction  $i_f^!(\cL[d_X]) = \cL|_{U_f}[d_X-2k]$ 
doesn't introduce a sign, but the shift 
$[d_X]: \cL|_{U_f}[d_X-2k] \to \cL|_{U_f}[2(d_X-k)]$ introduces 
by \eqref{th2}  the sign $(-1)^{(d_X-2k)d_X}=(-1)^{d_X}$.
\item[(c)] Similarly, the restriction  
$i_f^!(\cL[2d_X]) = \cL|_{U_f}[2(d_X-k)]$ 
doesn't introduce a sign, but the shift 
$[d_X]: \cL[d_X]\to \cL[2d_X]$ introduces also by \eqref{th2} the sign $(-1)^{d_X\cdot d_X}=(-1)^{d_X}$.
\end{enumerate}
So for the comparison of our signatures and $L$-classes 
of twisted intersection cohomology complexes with those of Cappell-Shaneson, only the factor $(-1)^{d_X}$
shows up, together with the additional factor $(-1)^{d_X\cdot k}$ mentioned before.
This implies
\begin{enumerate}
    \item $L_*([IC_X(\cL),\alpha [d_X]])=L_*^{CS}([IC_X(\cL)[d_X],\alpha [2d_X]])$ in case
    $\dim(X)$ is even and $\cL$ is symmetric (so that only the case $k$ even is needed). In particular 
    $L_*([IC_X,\alpha_{can}])=L_*(X)$ is the Goresky-MacPherson L-class for $\dim(X)$ even.
    \item $L_*([IC_X(\cL),\alpha [d_X]])= L_*^{CS}([IC_X(\cL)[d_X], \alpha [2d_X]])$ in case
   $\dim(X)$ is odd and $\cL$ is symmetric (so that only the case $k$ odd is needed). In particular
    $L_*([IC_X,\alpha_{can}])= L_*(X)$ is the Goresky-MacPherson L-class for $d_X=\dim(X)$ odd.
  \item $L_*([IC_X(\cL), \alpha [d_X]])=L_*^{CS}([IC_X(\cL)[d_X], \alpha [2d_X]])$ in case
    $\dim(X)$ is even and $\cL$ is skew-symmetric (so that only the case $k$ odd is needed).   
   \item $L_*([IC_X(\cL), \alpha [d_X]])= -L_*^{CS}([IC_X(\cL)[d_X], \alpha [2d_X]])$ in case
   $\dim(X)$ is odd and $\cL$ is skew-symmetric (so that only the case $k$ even is needed). 
\end{enumerate}
    \end{example}

In particular, for $X$ smooth we have
$$L_*([\bR_X[\dim(X)],\alpha_{can}])=L^*(TX)\cap [X] = T_1^*(TX)\cap [X]\:,$$
with $L^*(TX)$ the Hirzebruch $L$-class of the tangent bundle $TX$ of $X$.
One then has the following commutativity of another part of diagram \eqref{diag1} as in \cite[(4)]{BSY}.
\begin{thm}[\cite{BSY}]\label{bsycom} For $X$ a compact complex algebraic variety one has the following commutative diagram of natural transformations:
\begin{equation}\label{diag3}
    \begin{tikzcd}
 K_0(Var/X)  \arrow{r}{T_{y*}} \arrow{d}{\sd_\bR } & H_*(X)\otimes \bQ[y]  \arrow{d}{y=1}  \\
 \Omega_\bR(X) \arrow{r}{L_*} & H_{*}(X)\otimes\bQ \:.
     \end{tikzcd}
\end{equation}
\end{thm}

\begin{remark}
Since in general  $\sd_\bR([id_X])\neq [IC_X]$, one also gets that 
$T_{1*}(X) \neq L_*(X)$ for arbitrary compact complex algebraic varieties. For a more detailed study of the difference between these two classes, see \cite{FPS2}. In fact, in \cite{FPS2} (see also \cite[Remark 1.2b]{FPS}) the authors define an $L$-class transformation (or a variant $\tilde{L}$ using a further isomorphism for an induced pairing $\tilde{S}_f$ introducing the sign $(-1)^{k(k-1)/2}$ in homology degree $2k$)
\[ L_*^{pair}:\Omega^{pair}_\bR(X) \to H_*(X) \otimes \bQ = H_{2*}(X;\bQ)\]
by using the induced non-degenerate pairing (cf. \cite[Section 3]{FPS2})
\[S_f:=S|_{X_f}:\cF_f \otimes \cF_f \to \bD_{X_f}\]
defined as follows. Since $i_f^*\bD_X=\bD_{X_f}[2k]$, we get the induced pairing $S_f$ shifted by $2k$ via
\begin{eqnarray*}
\cF_f \otimes \cF_f [2k]&=&
\bR_{X_f}[k]\otimes \bR_{X_f}[k] \otimes \cF_f \otimes \cF_f \\
&\simeq &
\bR_{X_f}[k]\otimes \cF_f \otimes \bR_{X_f}[k]\otimes \cF_f \\ 
&=& (\cF \otimes \cF)|_{X_f} 
 \longrightarrow i_f^*\bD_X=\bD_{X_f}[2k] 
\end{eqnarray*}
Note that the second isomorphism introduces the  sign $(-1)^{ik}$ on $(\cF_f)^i \otimes (\cF_f)^j$. 
\end{remark}

\begin{example}\label{ex-twIC2} Let $\cF=IC_X(\cL)$ be a twisted intersection cohomology complex on a variety $X$ of pure dimension $d_X=\dim(X)$, with $\cL$ a self-dual local system with duality $\alpha: \cL\to \cL^{\vee}$ on an open Zariski-dense smooth subset $U$ of $X$ (given as the top-dimensional strata of an algebraic Whitney stratification of $X$). Then 
$$i_f^*(IC_X(\cL))[-k] = IC_{X_f}(\cL|_{U_f})$$
is also a twisted intersection cohomology complex,
but the induced pairing  of $$\cL|_{U_f}[d_X-k]\otimes \cL|_{U_f} [d_X-k]\to \bR_{U_f}[2(d_X-k)]$$ is multiplied with
the sign $(-1)^{(d_X-k)\cdot k}$. This uniquely fixes the 
induced  duality of $$i_f^*(IC_X(\cL))[-k] \:,$$ so that also here the sign $(-1)^{(d_X-k)\cdot k}$ shows up.
\end{example}

We can now show the following.
\begin{thm}
    With the above conventions and notations, the two $L$-class transformations
    $$L_*, L^{pair}_* \circ \iota : \Omega_\bK(X) \to H_{2*}(X;\bQ)$$
    coincide, i.e., $L_*= L^{pair}_* \circ \iota$.
\end{thm}

\begin{proof}
Via \eqref{cob}, $\Omega_\bK(X)$ is generated by twisted intersection cohomology sheaves $IC_Z(\cL)$ with $Z$ irreducible and $\cL$ a generically defined local system with a self-duality isomorphism (see, e.g. \cite[Corollary 2.13]{SW}, 
\cite[Corollary 7.4]{Yo} or \cite[Proposition 1.1c]{FPS}). So it is enough to check the equality for these complexes.
By functoriality of both $L$-class transformations for the closed inclusion $Z\hookrightarrow X$, we can even assume $Z=X$.
This follows then from Example \ref{ex-twIC1} and Example 
\ref{ex-twIC2}, since for both transformations the sign
$(-1)^{(d_X-k)\cdot k}$ shows up for the change of the self-duality isomorphism, needed for the calculation of the corresponding signature defining the $L$-class in homology degree $2k$.
\end{proof}

\begin{remark}
    If in the above construction of the $L$-class transformation via duality isomorphisms one uses the second possible sign convention (i.e., shifting in the first factor out of $Rhom$ is an equality, and shifting in the second factor out produces a sign isomorphism), this would 
    change the $L$-class in homology degree $2k$ by the factor $(-1)^{k(k-1)/2}$. One would still get a functorial $L$-class, but not satisfying the usual normalization on the $IC$-complex fitting with the Goresky-MacPherson $L$-classes. 
\end{remark}

\subsection{Known cases of the characteristic class version of the Hodge index theorem}\label{known}

By applying the transformation $L_*$ to both sides of formula \eqref{eqmain}, and using the identity $$T_{1*}=
L_* \circ \sd_\bR$$ from Theorem \ref{bsycom}, one gets the proof of Conjecture \ref{BSY} in the case of $X$ a compact complex algebraic variety which is a rational homology manifold; see Theorem \ref{pfs} and \cite{FPS}.\\

Prior to the recent papers \cite{FP,FPS} which deal with Conjecture \cite{BSY} in the rational homology manifold context, positive results in this context have been also obtained in \cite{CMSS}, for certain complex hypersurfaces with isolated singularities; in \cite{CMSS2}, for global quotients $X= Y/G$, with $Y$ a projective $G$-manifold and $G$ a finite group of agebraic autormorphisms; in \cite{MS1}, for projective simplicial toric varieties; and in \cite{Ban1}, for certain normal projective complex $3$-folds with at worst canonical singularities. Conjecture \cite{BSY} has been also proved recently in \cite{BSW} for all Schubert varieties in a Grassmannian, this being the first case of non-rational homology manifolds where the conjecture is known.


\section{Some non-rational homology cases of the conjecture}\label{secnew}
Let $X$ be a compact complex algebraic variety of pure dimension $d_X$, and let
$IC_X^H$ be its intersection cohomology Hodge module, with $$IC'^H_X:=IC_X^H[-d_X].$$

\subsection{Property (H) and first examples}\label{secH} Since Conjecture \ref{BSY} can be formulated by asking that the transformations $L_* \circ \Pol$ and $MHT_{1*}$ give the same value on $[IC'^H_X]$, we introduce the following:
\begin{definition}
Say that $[M]\in K_0(\MHM(X))$ satisfies {\it property (H)} if 
\begin{align}
    MHT_{1*}([M])=(L_{*}\circ \Pol)([M])\:. \tag*{({H})}
\end{align}
\end{definition}

So Conjecture \ref{BSY} is equivalent to asking that $[IC'^H_X]$ satisfies property (H). In Theorem \ref{comparison-smooth} below we will see that the  property (H) holds for any polarizable variation of pure Hodge structures on a smooth variety $X$ (i.e., corresponding to a shifted pure smooth Hodge module).

As noted in Remark \ref{r11}, the commutativity of diagrams \eqref{diag2} and \eqref{diag3} yields the following.
\begin{corollary}
If $[M] \in \im(\chi_{Hdg}) \subset K_0(\MHM(X))$, then $[M]$ satisfies property (H). 
In particular,  the constant Hodge complex $[\bQ^H_X]=\chi_{Hdg}([id_X])$ satisfies property (H). 
\end{corollary}

We next recall the following well-known fact (see \cite[Section 2.3]{S}):
\begin{lemma}[\cite{S}]\label{point}
If $(V,S)$ is a pure polarized Hodge structure of weight $w$, with polarization $S$, one has
\[ \chi_1(V) = \begin{cases}
0 & {\rm for} \ w \ {\rm odd}, \\
\sigma(V) & {\rm for} \ w \ {\rm even},
\end{cases}
\]
where $\sigma(V)$ is the signature of the induced symmetric bilinear form $(-1)^{w/2}\cdot S$ on $V$. As usual, $\chi_y(V)$ is defined here in terms of the Hodge filtration on the complexification $V_\bC$ of $V$, by
\[\chi_y(V):=\sum_p \dim_\bC (Gr^p_F (V_\bC)) \cdot (-y)^p.\]
\end{lemma} 

\begin{remark} Note that this result  doesn't depend on  which of the two possible definitions of a polarization one is using, because they only  differ by the factor $(-1)^w$.
\end{remark}

\begin{corollary}\label{cpoint}
If $X=pt$ is a point, then any $[V] \in K_0(\MHM(pt))$ satisfies property (H).
\end{corollary}
\begin{proof} 
It suffices to check the claim for $(V,S)$ a pure polarized Hodge structure of weight $w$.
By definition, $MHT_{1*}([V])=\chi_1(V)$.

If $w$ is even, we get by the definition of $\Pol$ and Lemma \ref{point} that
\begin{align*}
   L_*\circ \Pol ([V]) =L_0((-1)^{w(w+1)/2}\alpha_{\bR})=L_0((-1)^{w/2}\alpha_{\bR})=\sigma((-1)^{w/2}S_{\bR})=\chi_1(V).
\end{align*}

If $w$ is odd, then $S_{\bR}$ is skew-symmetric,  hence its signature is trivial, and we also have $\chi_1(V)=0$ by Lemma \ref{point}, or more directly by a corresponding duality argument (see \cite[Section 2.3]{S}): $\chi_1(V)=\chi_1(V^{\vee})=
(-1)^w \cdot \chi_1(V)$.
\end{proof}

We also have the following.
\begin{proposition}\label{commute!}
If $f:X\to Y$ is a morphism between compact complex algebraic varieties and $[M]\in K_0(\MHM(X))$ satisfies property (H), then $f_*[M]=f_{!}[M]\in K_0(\MHM(Y))$ also satisfies property (H).
\end{proposition}

\begin{proof}
This is an easy consequence of the fact that the transformations $\Pol$, $MHT_{1*}$ and $L_{*}$ all commute with pushforwards for morphisms of compact varieties.
\end{proof}

We next recall that $K_0(\MHM(X))$ is a module over $K_0(\MHM(pt))=K_0(\MHS^p)$, with the product induced from the external product via the identification $X \simeq X \times pt$. Moreover, as shown in \cite{BSY}, the class transformation $MHT_{y*}$ is compatible with this module structure. We also need the following lemmas:

\begin{lemma}[Module property for signature and $L$-class]\label{mod}
    Let $X$ be a compact complex algebraic variety. Let $(\cF,\alpha)$ be a self-dual  complex of real vector spaces on $X$, and $(V,\beta)$ a finite dimensional real self-dual vector space (viewed on a point space). Then
    \begin{equation}\label{modL}
       L_*(\cF \boxtimes V, \alpha \boxtimes \beta) = \sigma(V,\beta) \cdot L_*(\cF,\alpha).
    \end{equation}
\end{lemma}

\begin{proof}
For the proof let more generally $(\cF,\alpha)$ be a   complex of real vector spaces on $X$,
which is self-dual with respect to $\cD_X[2m]$ (for $m\in \bZ$).
Using the K\"unneth formula, together with the fact the Verdier duality commutes with $\boxtimes$ (see e.g., \cite[Section 3.7]{W}),
we get that the duality isomorphism induced from $\alpha \boxtimes \beta$ on the vector space 
\[ H^{-m}(\cF \boxtimes V)\simeq H^{-m}(\cF) \otimes V
\]
agrees with the tensor product of the ones induced from $\alpha$ on $H^{-m}(\cF)$ and, resp., $\beta$.
Therefore, the multiplicativity of the signature implies for the pushforward under the constant map $a:X \to pt$ that
\[
\sigma(a_*[\cF \boxtimes V, \alpha \boxtimes \beta])=   \sigma(a_*[\cF,\alpha]) \cdot \sigma(V,\beta).
\]
Applying this property to the construction of $L$-classes, yields the desired assertion (see also
 \cite[Proposition 8.2.20]{Ban0}).
\end{proof}

\begin{lemma}\label{modcomp}
The composition $L_* \circ \Pol$ is compatible with the $K_0(\MHM(pt))$-module structure on $K_0(\MHM(X))$, i.e., for $[a] \in K_0(\MHM(pt))$ and $[M]\in K_0(\MHM(X))$ one has:
\[ (L_{*}\circ \Pol)([M]\cdot [a])= (L_{*}\circ \Pol)([M]) \boxtimes (L_{*}\circ \Pol)([a]).\]
\end{lemma}

\begin{proof}
For $L$-classes, this property follows from Lemma \ref{mod}.

For  $\Pol$, it suffices to show the assertion for pure objects, i.e., assume $M_1\in \MH(X,w_1)^{(p)}$ is a pure polarizable Hodge module of weight $w_1$ on $X$, and $M_2\in \MH(pt,w_2)^{(p)}$ is a pure polarizable Hodge module of weight $w_2$ on $pt$. Then
\[
M_1\boxtimes M_2 \in \MH(X,w_1+w_2)^{(p)},
\]
and if $S_i$ are the corresponding polarizations, then $S_1\boxtimes S_2$ gives a polarization of $M_1\boxtimes M_2$. Let $\alpha_1,\alpha_2$ be the corresponding self-duality isomorphisms.
By the definition of Pol,
\[\Pol(M_1)=[(-1)^{\frac{w_1(w_1+1)}{2}}\alpha_1], \ \ \Pol(M_2)=[(-1)^{\frac{w_2(w_2+1)}{2}}\alpha_2]\]
\[\Pol(M_1\boxtimes M_2)=[(-1)^{\frac{(w_1+w_2)(w_1+w_2+1)}{2}}\alpha_1\boxtimes \alpha_2].\]

So the difference between $\Pol(M_1\boxtimes M_2)$ and $\Pol(M_1)\boxtimes \Pol(M_2)$ is given by the sign $(-1)^{w_1w_2}$. If $w_2$ is even, then no sign appears. Assume now that $w_2$ is odd. Then $L_*(\alpha_2)=\sigma(S_2)=0$ as in Lemma \ref{point}, hence by the module property \eqref{modL} of $L_*$ (see also \cite[Proposition 8.2.20]{Ban0}), we get
\[L_*([\alpha_1\boxtimes \alpha_2]) =L _*([\alpha_1])\boxtimes L_*([\alpha_2]) =0,\]
and therefore $$(L_*\circ \Pol)(M_1\boxtimes M_2)= (-1)^{w_1} (L_*\circ \Pol)(M_1)\boxtimes (L_*\circ \Pol)(M_2)=0\:.$$
\end{proof}

\begin{remark}\label{rem-Lpair}
 The proof of the module property from Lemma \ref{mod} also applies to the $L$-class transformation $L_*^{pair}$, so that Lemma \ref{modcomp} and its proof also apply to $L_*^{pair}$.
\end{remark}

\begin{proposition}\label{sumlemma}
Let $$[M]=\sum_i  [M_i] \cdot [a_i] \in K_0(\MHM(X\times pt))\simeq K_0(\MHM(X)),$$ with $[a_i] \in K_0(\MHS^p)$, so that every $[M_i]\in K_0(\MHM(X))$ satisfies property (H). Then $[M]$ satisfies property (H).
\end{proposition}

\begin{proof}
Since the transformations $L_* \circ \Pol$ and $MHT_{1*}$ respect the module structure, we get that
\begin{equation*}\label{moduleoverpoint}
\begin{split}
     MHT_{1*}([M_i]\cdot [a_i]) &=   MHT_{1*}([M_i]) \boxtimes MHT_{1}([a_i])   \\
     &=  (L_{*}\circ \Pol)([M_i]) \boxtimes (L_{*}\circ \Pol)([a_i])  \\
     &=(L_{*}\circ \Pol)([M_i]\cdot [a_i]) \in H_*(X\times pt)\otimes \bQ \cong H_*(X) \otimes \bQ,
     \end{split}
\end{equation*}
where the second equality uses Corollary \ref{cpoint} and the fact that $[M_i]$ satisfies property (H). The claim follows.
\end{proof}

\begin{corollary}\label{submodule}
If $[M]\in K_0(\MHM(X))$ belongs to the $K_0(\MHM(pt))=K_0(\MHS^p)$-submodule generated by $\im(\chi_{Hdg}) \subset K_0(\MHM(X))$,
then $[M]$  satisfies  property (H). 
\end{corollary}

\begin{remark}\label{module-functorial}
The   $K_0(\MHM(pt))=K_0(\MHS^p)$-submodule generated by $\im(\chi_{Hdg}) \subset K_0(\MHM(-))$ is by  \cite[Section 4.2]{S} stable under the transformations $f_!$, $\boxtimes$, and $g^*$.
\end{remark}

 \begin{example}[$\bQ$-homologically isolated singularities]\label{iso}
Let $X$ be as above a compact complex algebraic variety of pure dimension, and denote by  $\Sigma_X$ the non-rational homology locus of $X$ (i.e., the smallest  closed subset of $X$ such that $X\setminus \Sigma_X$ is a rational homology manifold).
Consider  the natural comparison morphism
\begin{equation}\label{comp} \bQ^H_X \longrightarrow IC'^H_X \end{equation}
and let $M_X$ denote its cone in $D^b\MHM(X)$ (up to isomorphism). Note that the support of $M_X$ is a closed algebraic subset, agreeing with 
 the non-rational homology locus $\Sigma_X\subseteq X_{sing}$  of $X$, which is contained in the singular locus $X_{sing}$ of $X$.  If $X$ has $\Q$-homologically isolated singularities (i.e., $\dim (\Sigma_X)=0$), then the Conjecture \ref{BSY} holds (cf. also \cite{FPS,FPS2}). Indeed, since $M_X$ is supported on points, and 
\[
[IC_X'^H]=[\bQ^H_X]+[M_X],\] the assertion follows from Proposition  \ref{sumlemma} (or Corollary \ref{submodule}). This applies in particular in case $X$ has only isolated singularities (i.e., $\dim X_{sing}=0$).
\end{example}


\subsection{Coincidence in the top non-rational homology locus degree}

Following up on Example \ref{iso}, let $\delta_X=\dim(\Sigma_X)$ be the dimension of the non-rational homology locus $\Sigma_X$  of $X$. Recall that the cone $M_X$ of the comparison morphism \eqref{comp} has support $\Sigma_X$, so in view of \eqref{eqmain} one has that 
\[ IT_{1,k}(X)=T_{1,k}(X)=L_k(X) \in H_{k}(X) \otimes \bQ \ \ {\rm if} \ \ k> \delta_X.\]
As
\[ [IC'^H_X] = [\bQ_X^H] + [M_X] \in K_0(\MHM(X)),\]
and since $[\bQ^H_X]$ satisfies property (H), one gets that 
\begin{align*}
    L_*(X)-IT_{1*}(X)=(L_{*}\circ \Pol)([M_X])-MHT_{1*}([M_X]).
\end{align*}
Hence Conjecture \ref{BSY} can be further reduced to showing that the cone $[M_X]$ satisfies property (H). As seen in Example \ref{iso}, this is the case when $\delta_X=0$, and it was also recently proved in \cite[Proposition 3.3]{FPS2} for $\delta_X\leq 1$. Here we generalize the latter statement as follows:
\begin{proposition} In the above notations, we have the equality
    \[ IT_{1,\delta_X}(X)=L_{\delta_X}(X) \in H_{\delta_X}(X) \otimes \bQ
  =H_{2\delta_X}(X;\bQ)  \:.\]
\end{proposition}
This is a consequence of the following result.
\begin{thm}\label{t410}
    Let $M$ be a pure ($\bQ$-)Hodge module on a compact, pure-dimensional complex algebraic variety $\Sigma$ of dimension $\delta$. Then the degree $\delta$-parts of the classes 
    $(L_{*}\circ \Pol)([M])$ and $MHT_{1*}([M])$ coincide in $H_{\delta}(\Sigma) \otimes \bQ$.
\end{thm}
Before proving the theorem, we need the following.
\begin{lemma}\label{lem411}
    If $(\bV,S)$ is a pure polarizable variation of $\bQ$-Hodge structures on a connected smooth complex algebraic variety $\Sigma$ of pure dimension $\delta$, then the degree $\delta$ part $MHT_{1,\delta}([\bV^H])$ of $MHT_{1*}([\bV^H])$ is $\chi_1(V) \cdot [\Sigma]$, where $\bV^H$ is the (shifted) Hodge module on $\Sigma$ associated to $\bV$, and $V$ is the stalk of the variation $\bV$. 
\end{lemma}

\begin{proof}
Let $\cV:=\bV \otimes \cO_\Sigma$ be the flat bundle with connection, whose sheaf of horizontal sections is $\bV \otimes \bC$. The bundle $\cV$ comes equipped with its Hodge (decreasing) filtration by holomorphic sub-bundles $\cF^p$ (which satisfy Griffiths' transversality). Let
\begin{equation}\label{chi-y-class}
 \chi_y(\cV):=\sum_p [Gr^p_\cF \cV] \cdot (-y)^p \in K^0(\Sigma)[y^{\pm 1}]  
\end{equation}
 be the $K$-theory $\chi_y$-genus of $\cV$, with $K^0(\Sigma)$ denoting the Grothendieck group of algebraic vector bundles on $\Sigma$.
The twisted Chern character $ch^*_{(1+y)}$ is defined as
\[ ch^*_{(1+y)}(E):=\sum_{j=1}^{\rk(E)} e^{\beta_j (1+y)},\]
for $\{\beta_j\}$ the Chern roots of the bundle $E$. In particular, in cohomological degree $2k$ one has the identity:
\begin{equation}\label{2k}  ch^*_{(1+y)}(E)_{2k}=(1+y)^k \cdot ch^*(E)_{2k}. \end{equation}

With the above notation, the following Atiyah-Meyer type formula 
was obtained in \cite[Theorem 11]{CLMS} (see also \cite[Theorem 4.1]{MS0} for the normalized version discussed here):
\begin{equation}\label{clms}
    MHT_{y*}([\bV^H])=\left( ch^*_{(1+y)}(\chi_y(\cV)) \cup T_y^*(\Sigma) \right) \cap [\Sigma].
\end{equation}
By setting $y=1$ in \eqref{clms}, we see that in order to compute the degree $\delta$ part 
$$ MHT_{1,\delta}([\bV^H]) \in H_\delta(\Sigma) \otimes \bQ=H_{2\delta}(\Sigma;\bQ)\:,$$
we need to compute the cohomological degree $0$ part 
\[ ch^0_{(2)}(\chi_1(\cV)) \cup T_1^0(\Sigma)\in H^0(\Sigma;\bQ)\:,\]
that is, the product of the cohomological degree $0$ parts of each of the classes $ch^*_{(2)}(\chi_1(\cV))$ and $T_1^*(\Sigma)$.
First, for a complex vector bundle $E$, we have
\[ ch^0_{(2)}(E)=ch^0(E)=\rk(E).\]
Since the rank of $E$ is just the complex dimension of its fiber, and since the fiber of 
$Gr^p_\cF\cV$ is $Gr^p_F V_\bC$, with $V_\bC$ the complexification of the stalk $V$ of $\bV$, we then see that
\[ch^0_{(2)}(\chi_1(\cV))=\chi_1(V). \]
Moreover, since the power series defining the cohomology class $T_1^*(\Sigma)$ is normalized, the cohomological degree $0$ part of this class is $1=T_1^0(\Sigma)$.
The assertion of the lemma follows.
\end{proof}

We can now return to Theorem \ref{t410}.
\begin{proof}[Proof of Theorem \ref{t410}]
  By the strict support decomposition of pure Hodge modules, it suffices to assume that $M=IC_\Sigma^H(\bV)$, for $\bV$ a polarized variation of Hodge structures of weight $w$ on a (connected) Zariski-open subset $U$ of the regular part $\Sigma_{reg}$ of $\Sigma$, with $\Sigma$ irreducible of dimension $\delta$ and $\Sigma_{reg}$ connected. Then $M$ has weight $w+\delta$.  
Let $S^H$ be the polarization on $M$ corresponding uniquely to the polarization $S$ of $\bV$ (as discussed in Remark \ref{ext-pol} for the untwisted case). Then one has as in \eqref{Sai} that
\[ S^H= (-1)^{\delta(\delta-1)/2}S_{can},\]
with $S_{can}$ the perfect pairing of $IC_\Sigma(\bV)$ induced by 
\[S: \bV\otimes \bV = \cH^{-2\delta}(\bV[\delta]\otimes \bV[\delta])\to \bQ_U.\]
Then, using the corresponding self-duality isomorphisms, we get 
\begin{eqnarray*}
    (L_* \circ \Pol)([M]) &=& (-1)^{(w+\delta)(w+\delta+1)/2} \cdot L_*([\alpha_{S^H}])\\
    &=& (-1)^{(w+\delta)(w+\delta+1)/2} \cdot (-1)^{\delta(\delta-1)/2} \cdot L_*([\alpha_{S_{can}}]).
\end{eqnarray*}
For $p \in U$ and $V=\bV_p$ the stalk of $\bV$ at $p$, let $S_p$ be the corresponding polarization on $V$. It then follows from \cite[page 543]{CS} (see also \cite[page 197]{Ban0}) that 
\[
L_\delta([\alpha_{S_{can}}])=\sigma(S_p) \cdot [\Sigma] \in H_{\delta}(\Sigma) \otimes \bQ=H_{2\delta}(\Sigma;\bQ).
\]
Note that this identity holds even with our modified 
$L$-class $L_*$ (compared to the Cappell-Shaneson $L$-class $L^{CS}_*$), since by Example \ref{ex-twIC1} an additional minus sign only comes in when already  $\sigma(S_p)=0$.
Altogether, the degree $\delta$-part of the class 
    $(L_{*}\circ \Pol)([M])$
    equals 
\begin{equation}\label{sign1}
 (-1)^{(w+\delta)(w+\delta+1)/2} \cdot (-1)^{\delta(\delta-1)/2} \cdot \sigma(S_p) \cdot [\Sigma].       
    \end{equation}
    Let us also note that if $w$ is odd, then $S_p$ is anti-symmetric, so $\sigma(S_p)=0$.

    On the other hand, for calculating the degree $\delta$-part of the class $MHT_{1*}([M])$, using the open restriction isomorphism
    $$H_{\delta}(\Sigma)\otimes \bQ \simeq 
    H_{\delta}(U)\otimes \bQ,$$
    we can assume that $\Sigma=U$ is smooth, hence $M=\bV^H[\delta]$, with $\bV^H[\delta]$ the (smooth) Hodge module on $\Sigma$ associated to the pure polarizable variation of $\bQ$-Hodge structures  $\bV$. Then, in view of Lemma \ref{lem411}, the degree $\delta$-part of the class $MHT_{1*}([M])$ equals 
    $$(-1)^\delta \cdot \chi_1(V) \cdot [\Sigma]\:,$$
    where $\chi_1(V)$ is calculated as in Lemma \ref{point}, i.e., it equals $0$ if $w$ is odd, and it is equal to $(-1)^{w/2}\cdot \sigma(S_p)$ 
     if $w$ is even. Hence the degree $\delta$-part of the class $MHT_{1*}([M])$ equals
\begin{equation}\label{sign2} 
    (-1)^\delta \cdot (-1)^{w/2} \cdot \sigma(S_p) \cdot [\Sigma] \:,
    \end{equation}
    which by definition vanishes if $w$ is odd.
    It is now immediate to check that (the two signs in) \eqref{sign1} and \eqref{sign2} agree when $w$ is even, thus completing the proof.
\end{proof}

As a consequence, we get immediately the following result from \cite[Proposition 3.3]{FPS2}:
\begin{corollary}\label{cor-1dim}
    Conjecture \ref{BSY} holds for a compact complex algebraic variety $X$ if the dimension $\delta_X$ of the non-rational homology locus of $X$ is one.
\end{corollary}

\begin{proof}
 By Theorem \ref{t410}, only the equality of both characteristic classes in degree $0$ needs to be checked. Here we can assume $X$ is connected with
 $$deg: H_0(X;\bQ)\simeq \bQ \:,$$
 so that this claim follows from \eqref{ihif}.
\end{proof}

\begin{example}\label{ex-1dim}
 Let $X$ be a compact complex algebraic variety of pure dimension, whose singular locus $X_{sing}$ is one-dimensional (e.g., $X$ is in addition  normal of complex dimension three). Then Conjecture \ref{BSY} holds for $X$.
 In particular, the conjecture holds for $X$ a compact complex algebraic variety of pure dimension three (via normalization and Remark \ref{r14}).
\end{example}

\subsection{Comparison of Atiyah-Meyer type formulae for smooth pure Hodge modules}\label{subsec-AM}

Let us recall from the proof of Lemma \ref{lem411}
the following Atiyah-Meyer type formula \eqref{clms}
from \cite[Theorem 11]{CLMS} (see also \cite[Theorem 4.1]{MS0} for the normalized version discussed here).

If $(\bV,S)$ is a pure polarizable variation of $\bQ$-Hodge structures of weight $w$ on a smooth complex algebraic variety $X$ of pure dimension $d_X$, then 
\begin{equation}\label{clms2}
    MHT_{1*}([\bV^H])=\left( ch^*_{(2)}(\chi_1(\cV)) \cup T_1^*(X) \right) \cap [X]\:,
\end{equation}
where $\bV^H$ is the (shifted) smooth Hodge module on $X$ associated to $\bV$. Here
\begin{equation}\label{chi-1-class}
 \chi_1(\cV):=\sum_p (-1)^p\cdot [Gr^p_\cF \cV]  \in K^0(X)
\end{equation}
is the corresponding K-theoretical $\chi_1$-class in the Grothendieck group $K^0(X)$ of algebraic vector bundles on $X$, and $\cV:=\bV \otimes \cO_X$  is the flat bundle with connection, whose sheaf of horizontal sections is $\bV \otimes \bC$. The bundle $\cV$ comes equipped with its Hodge (decreasing) filtration by holomorphic sub-bundles $\cF^p$ (which satisfy Griffiths' transversality).
In the following it is enough to consider the underlying topological class $For( \chi_1(\cV))\in K^0_{top}(X)$ in the Grothendieck group of topological complex vector bundles on $X$,
with 
$$ch^*_{(2)}= ch^*\circ \Psi^2 $$
the composition of the topological Chern character with the second Adams operation, so that
\begin{equation}\label{clms3}
    MHT_{1*}([\bV^H])=\left( ch^*\circ \Psi^2(For(\chi_1(\cV))) \cup T_1^*(X) \right) \cap [X]\:.
\end{equation}
Here $\Psi^2$ is fixed by the splitting principle together with $\Psi^2(\cL)=\cL\otimes \cL$ for a complex line bundle $\cL$, so that 
$$ ch^*\circ \Psi^2(\cL) = 1+e^{c^1(\cL\otimes \cL)} =
1+e^{2c^1(\cL)} \:.$$
The underlying topological complex vector bundle of $\cV$  has a natural real structure so that
as a topological complex vector bundle one gets an orthogonal decomposition (as in \cite[Corollary 3.8]{S}):
$$\cV = \bigoplus_{p+q=w} \cH^{p,q} \quad\text{, \ with} \quad \cH^{p,q} = F^p\cV \cap \overline{ F^q\cV} = \overline{\cH^{q,p}}\:,$$
and
\begin{equation}
For(\chi_1(\cV)) = \bigoplus_{p\: even,q}\: [\cH^{p,q}]
- \bigoplus_{p\: odd,q}\: [\cH^{p,q}]\in K^0_{top}(X)\:.
\end{equation}
Moreover, $For(\chi_1(\cV))$ fits by \cite[Corollary 3.8]{S}
with a corresponding class  
$$[(\bV,(-1)^{w(w+1)/2}\cdot S)]_K$$  introduced by Meyer \cite[Section 4]{Mey}
for the underlying local system $\bV$ with the pairing 
$$(-1)^{w(w+1)/2}\cdot S\:.$$ 
We follow here the more recent 
presentation given in \cite[Section 2]{RW}.\\

Let $L$ be a self-dual local system of real vector spaces on $X$ (here $X$ could be an oriented manifold) given by a non-degenerate pairing $S': L\otimes L\to \bR_X$. Then one defines  a K-theoretical {\it signature class}
$$[(L, S')]_K \in K^0_{top}(X)$$
as follows. On the associated flat real vector bundle $\cL$ one can find a Riemannian metric $\langle -, -\rangle$
and a linear operator $A: \cL\to \cL$, with
$S'(-,-)=\langle -, A -\rangle$ and $A^2=\pm id$.
\begin{enumerate}
 \item Assume $S'$ is symmetric. Then $A^*=A$ and $A^2=id$, so that $\cL=\cL^+\oplus \cL^-$ with $A\vert_{\cL^{\pm}}=\pm id$,
 and 
 $$[(L, S')]_K:= [\cL^+\otimes_\bR \bC] - [\cL^-\otimes_\bR \bC] \in K^0_{top}(X)\:.$$
\item  Assume $S'$ is skew-symmetric. Then $A^*=-A$ and $A^2=-id$, so that $A$ defines a complex structure on $\cL$, with
 $$[(L, S')]_K:= [\cL,-A] - [\cL,A] =
[(\cL\otimes_\bR \bC)^-] - [(\cL\otimes_\bR \bC)^+] \in K^0_{top}(X)\:.$$
Here $\cL\otimes_\bR \bC=(\cL\otimes_\bR \bC)^+\oplus 
(\cL\otimes_\bR \bC)^-$, with 
$A\vert_{(\cL\otimes_\bR \bC)^{\pm}}=\pm i\cdot id$.
 \end{enumerate}

Let now $(\bV,S)$ be a pure polarizable variation of $\bQ$-Hodge structures of weight $w$ on a smooth complex algebraic variety $X$ of pure dimension $d_X$, with 
$(L,S):=(\bV_\bR,S_\bR)$.
\begin{enumerate}
    \item Assume $w$ is even. Then the pairing $S(-,C(-))=
    (-1)^{w/2}\cdot S(-,(-1)^{w/2}\cdot C(-))$ is symmetric and positive definite, with $C$ the Weil operator (defined over $\bR$) acting by $i^{p-q}$ on $\cH^{p,q}$. Here we can take 
    $$\langle -, -\rangle=(-1)^{w/2}\cdot 
    S(-,(-1)^{w/2}\cdot C(-))$$ for the {symmetric} pairing ${S'=}(-1)^{w/2}\cdot S$, with $A=(-1)^{w/2}\cdot C$:
    $$\langle -, (-1)^{w/2}\cdot C(-)\rangle =(-1)^{w/2}\cdot S(-,-)\:.$$
    Then $A=(-1)^{w/2}\cdot C$ acts by $(-1)^p$ on $\cH^{p,q}$, {and $A^2=id$}, so that
    $$For(\chi_1(\cV)) =[(\bV,(-1)^{w/2}\cdot  S)]_K{ =[(\bV,(-1)^{w(w+1)/2}\cdot  S)]_K }\in K^0_{top}(X)\:.$$
   \item Assume $w$ is odd. Then the pairing 
     $(-1)^{(w+1)/2}\cdot S(-,(-1)^{(w+1)/2}\cdot C(-))$ is symmetric and positive definite, with $C$ the Weil operator (defined over $\bR$) acting by $i^{p-q}$ on $\cH^{p,q}$ as before. 
Here we can take 
    $$\langle -, -\rangle=(-1)^{(w+1)/2}\cdot 
    S(-,(-1)^{(w+1)/2}\cdot C(-))$$ for the {skew-symmetric} pairing ${S'=}(-1)^{(w+1)/2}\cdot S$, with 
    $A=-(-1)^{(w+1)/2}\cdot C = (-1)^{(w-1)/2}\cdot C $:
  $$\langle -, (-1)^{(w-1)/2}\cdot C (-)\rangle=(-1)^{(w+1)/2}\cdot S(-,-) \:.$$   
Then 
    $A=(-1)^{(w-1)/2}\cdot C$ acts by $(-1)^p\cdot i^{-1}$ on $\cH^{p,q}$,  and $A^2=-id$, so that
    $$For(\chi_1(\cV)) =[(\bV,(-1)^{(w+1)/2}\cdot  S)]_K=[(\bV,(-1)^{w(w+1)/2}\cdot  S)]_K  \in K^0_{top}(X)\:.$$ 
\end{enumerate}
    
    Assume now that $X$ is in addition compact
    (here $X$ could be a closed oriented manifold of even real dimension $2d_X$). Then Meyer shows in \cite[Theorem II 4.1]{Mey} (using the formulation of \cite[Equation (2.1)]{RW}) the {\it signature formula}
    \begin{equation}\label{sig-Meyer}
     \sigma(H^{d_X}(X,\bV))=\int_{[X]} \:
     ch^*\circ \Psi^2([(\bV,(-1)^{w(w+1)/2}\cdot S)]_K)\cup L^*(X)
    \end{equation}
    for the pairing induced by $S'=(-1)^{w(w+1)/2}\cdot S$
    on $L=\bV$.

\begin{remark}
Here we use (as in \cite[Section 2]{RW}) the signature of the pairing
\begin{equation*}\begin{CD}
H^{d_X}(X,L)\times H^{d_X}(X,L)@> \cup >> H^{2d_X}(X,L\otimes L)
@> S' >> H^{2d_X}(X,\bR_X) @> \int_{[X]} >> \bR \:,
\end{CD}\end{equation*}
whereas Meyer in \cite[Theorem II 4.1]{Mey} introduces for the case $d_X$ odd and $L$ skew-symmetric an additional minus sign for this pairing (see \cite[pp. 12-13]{Mey}).
This is also the reason why Meyer uses in this case the negative K-class $\:-[(L, S')]_K= [\cL,-A'] - [\cL,A'] $ for the operator $A'=-A$, so that 
$S'(-,A'(-))=\langle -, -\rangle$
is a Riemannian metric $\langle -, -\rangle$.
(In the other case $d_X$ even and $L$ symmetric one has $A'=A$.)
\end{remark}

    Since even shifts do not introduce signs, the 
{\it signature formula} \eqref{sig-Meyer}  
    implies the topological {\it Atiyah-Meyer type formula}
    (see also \cite{BCS} for a more general version for a singular $X$):
\begin{equation}\label{top-A-Meyer}
  L_*^{CS}([\bV[2d_X], \alpha [2d_X]])   = \left(
     ch^*\circ \Psi^2([(\bV,(-1)^{w(w+1)/2}\cdot S)]_K)\cup L^*(X)\right) \cap [X]
    \end{equation}
 for the induced duality pairing on $IC_X(\bV)[d_X]=\bV[2d_X]$ given by the pairing
 $(-1)^{w(w+1)/2}\cdot S$ on $\bV$.
 
This implies the following. 
\begin{thm}\label{comparison-smooth}
Let $(\bV,S)$ be a pure polarizable variation of $\bQ$-Hodge structures of weight $w$ on a compact smooth complex algebraic variety $X$ of pure dimension $d_X$, 
with $\bV^H$ the (shifted) smooth Hodge module on $X$ associated to $\bV$. Then
\begin{equation}
    MHT_{1*}([\bV^H])  =  
    (L_* \circ \Pol)([\bV^H]) \:.  
 \end{equation}
\end{thm}

\begin{proof}
 Let $\alpha_{S_{can}}$ be the duality induced on $\bV[d_X]$ via the polarization $S$. Then by definition  
\begin{eqnarray*}
    (L_* \circ \Pol)([\bV^H]) 
    &=& (-1)^{(w+d_X)(w+d_X+1)/2} \cdot (-1)^{d_X(d_X-1)/2} \cdot (-1) ^{d_X}\cdot L_*([\alpha_{S_{can}}])\\
    &=& (-1)^{w\cdot d_X} \cdot (-1)^{w(w+1)/2}\cdot 
     L_*([\alpha_{S_{can}}])\\
      &=& (-1)^{w\cdot d_X} \cdot
      L_*\left( (-1)^{w(w+1)/2}\cdot [\alpha_{S_{can}}]\right)\:.
\end{eqnarray*} 
with the sign $(-1)^{d_X}$ coming from $[\bV^H]=(-1)^{d_X}\cdot [\bV^H[d_X]]$. Using Example \ref{ex-twIC1},
we get by the Atiyah-Meyer type formulae \eqref{clms2}
and \eqref{top-A-Meyer} for the induced duality $\alpha$ coming from $(-1)^{w(w+1)/2}\cdot S$:

\begin{eqnarray*}
    (L_* \circ \Pol)([\bV^H]) 
    &=&  L_*^{CS}([\bV[2d_X], \alpha [2d_X]])\\
    &=& \left(
     ch^*\circ \Psi^2([(\bV,(-1)^{w(w+1)/2}\cdot S)]_K)\cup L^*(X)\right) \cap [X]\\
     &=& \left( ch^*\circ \Psi^2(For(\chi_1(\cV))) \cup T_1^*(X) \right) \cap [X]\\
     &=&  MHT_{1*}([\bV^H])\:.
\end{eqnarray*} 
\end{proof}

Following up on Example \ref{iso}, let $\Sigma_X$
be the non-rational homology locus   of $X$. Recall that the cone $M_X$ of the comparison morphism \eqref{comp} has support $\Sigma_X$.
As
\[ [IC'^H_X] = [\bQ_X^H] + [M_X] \in K_0(\MHM(X)),\]
and since $[\bQ^H_X]$ satisfies property (H), one gets as before that 
\begin{align*}
    L_*(X)-IT_{1*}(X)=(L_{*}\circ \Pol)([M_X])-MHT_{1*}([M_X]).
\end{align*}
Hence Conjecture \ref{BSY} can be further reduced to showing that (as before) the cone $[M_X]$ satisfies property (H).

\begin{corollary}
Let $X$ be a compact and  pure-dimensional complex algebraic variety endowed with a complex algebraic Whitney stratification so that the non-rational homology locus   $\Sigma_X$ of $X$ is a disjoint union of {\it closed}  strata. Then Conjecture \ref{BSY} holds for $X$.
\end{corollary}

\begin{proof}
By our assumptions, $\rat(M_X)$ is supported on $\Sigma_X$
and constructible with respect to the  induced (Whitney) stratification of  $\Sigma_X$, which is just a disjoint union of closed complex algebraic submanifolds. Then all cohomology sheaves $\cH^j(\rat(M_X))$ of $\rat(M_X)$ are locally constant ($j\in\bZ$), so that all weight graded pieces of $H^j(M_X)$
are smooth and pure Hodge modules corresponding to
pure polarizable variations of $\bQ$-Hodge structures on $\Sigma_X$
(up to a shift). Then $[M_X]$ satisfies property (H) by an inductive application of Theorem \ref{comparison-smooth}.
\end{proof}

\subsection{(Virtually) stratum-wise constant complexes}
In order to apply the results of Section \ref{secH} to specific geometric situations, we first describe the Grothendieck class of a complex $M \in D^b\MHM(X)$, compare with \cite[Proposition 5.1.2]{MSS}:

\begin{proposition} \label{ccs}
For a complex algebraic variety $X$, fix $M \in D^b\MHM(X)$ with underlying bounded constructible complex $K:=\rat(M)\in D^b_c(X;\bQ)$.
Let $\cS=\{S\}$ be a complex algebraic stratification of $X$ so
that for any stratum $S\in\cS$, $S$ is smooth, ${\bar S}\setminus S$ is a union of
strata, and the sheaves $L_{S,\ell}:=\cH^{\ell}K|_S$ are local systems on $S$ for any $\ell \in \bZ$.
If $j_S:S \hookrightarrow X$ is the inclusion map of a stratum $S\in \cS$, then: 
\begin{equation} \label{f1} [M] 
=\sum_{S,\ell}\,(-1)^{\ell}\,\big[(j_S)_!L_{S,\ell}^H \big]
\in K_0(\MHM(X)), 
\end{equation}
where $L_{S,\ell}^H=H^{\ell + \dim (S)} (j_S)^* M[-\dim(S)]$ is the shifted smooth mixed Hodge module on $S$ with $L_{S,\ell}=rat(L_{S,\ell}^H)$. 

If, moreover, 
the stratification $\cS$ of $X$ can be chosen so that the local systems $L_{S,\ell}$ on $S$ are constant for any $S$ and $\ell$,  then 
\begin{equation}\label{f4b}
[M]=\sum_{S}\,\big[(j_{S})_!\bQ_S^H \big] \cdot i_s^*[M],
\end{equation}
where, for $s\in S$ fixed, $i_s: \{s\} \hookrightarrow X$ denotes the inclusion map.
\end{proposition}

\begin{proof}
Formula \eqref{f1} follows by induction on strata, using the standard attaching triangles, see also \cite{MSS}.

If the local system $L_{S,\ell}$ on $S$ is constant, it follows by rigidity (e.g., see \cite[Section 3.1]{CMS}) that the (admissible) variation of mixed Hodge structures $L^H_{S,\ell}$ on $S$ (corresponding to a smooth mixed Hodge module as in the statement) is the constant variation. This implies that
$$L^H_{S,\ell}\simeq  \bQ_S^H \boxtimes L^H_{S, s, \ell},$$
where $L_{S, s, \ell}$ is the stalk of $L_{S, \ell}$ at $s \in S$, which underlies the mixed Hodge structure $L^H_{S, s, \ell}$.
Using the $K_0(\MHM(pt))$-module structure of $K_0(\MHM(X))$, formula \eqref{f1} becomes:
\begin{equation}\label{f4}
\begin{split}
 [M]
=\sum_{S,\ell}\,(-1)^{\ell}\, \big[(j_{S})_!\bQ_S^H \big] \cdot [L^H_{S, s, \ell}]
=\sum_{S}\,\big[(j_{S})_!\bQ_S^H \big] \cdot [M_s]
=\sum_{S}\,\big[(j_{S})_!\bQ_S^H \big] \cdot i_s^*[M],
\end{split}
\end{equation}
with $M_s:=i_s^*M$.
\end{proof}

\begin{definition}\label{sc}
Following \cite{MS2}, we call $M \in D^b\MHM(X)$ a {\it stratum-wise constant} complex of mixed Hodge modules (with respect to an algebraic stratification $\cS$ as above)  if the cohomology sheaves of the underlying constructible complex $K=rat(M)\in D^b_c(X;\bQ)$ are constant along the strata of the stratification. More generally, we call a class $[M]\in K_0(\MHM(X))$ {\it virtually stratum-wise constant} if equality \eqref{f4b} holds.
\end{definition}

\begin{remark}
By Proposition \ref{ccs}, a stratum-wise constant complex of mixed Hodge modules $M\in D^b\MHM(X)$ represents a virtually stratum-wise constant class $[M]\in K_0(\MHM(X))$.
\end{remark}

\begin{example} \label{ex411}
If, in the notations of Proposition \ref{ccs}, all strata of the stratification $\cS$ of $X$ are simply connected, then all local systems $L_{S,\ell}$ are constant, and $M$ is  stratum-wise constant, hence
$[M]\in K_0(\MHM(X))$ is virtually stratum-wise constant. In particular, this applies to varieties which admit algebraic cellular stratifications (i.e., all strata are isomorphic to complex vector spaces). Such examples include:
\begin{itemize}
    \item Schubert varieties (i.e., closures of Schubert cells given by $B$-orbits) in generalized flag varieties $G/P$, for $G$ a  connected reductive complex algebraic group, $P$ a parabolic subgroup, and $B$ a Borel subgroup of $G$ contained in $P$.
    \item matroid Schubert varieties (e.g., see \cite[Corollary A.2]{Cr}).
\end{itemize}
\end{example}

\begin{example} If, in the notations of Proposition \ref{ccs}, all the local systems $L_{S,\ell}$ on $S$ underlie {\it unipotent} variations of mixed Hodge structures (i.e., the graded parts $Gr_k^W L_{S,\ell}$ of the weight filtration are constant, for all $k \in \bZ$), then $[M]\in K_0(\MHM(X))$ is virtually stratum-wise constant.
\end{example}

\begin{remark}\label{exvcc}
Denote by $[L^H] \in K_0(\MHM(pt))$ the Grothendieck class of a (polarizable) mixed Hodge structure $L^H$. Then, in the notations of Proposition \ref{ccs}, for a fixed stratum $S \in \cS$ of $X$, the class
$$ \big[(j_{S})_!\bQ_S^H \big] \cdot [L^H]=\big[(j_{S})_! L_S^H] =: [M_S],$$ where $L^H_S=\bQ_S^H  \boxtimes L^H$, is virtually stratum-wise constant, with $i^*_s[M_S]=[L^H]$ for $s\in S$  and zero for all other stalks. This shows that a class $[M]\in K_0(\MHM(X))$ is virtually stratum-wise constant if, and only if, it belongs to the (free) $K_0(\MHM(pt))$-submodule of $K_0(\MHM(X))$ generated by classes $\big[(j_{S})_!\bQ_S^H \big]=\chi_{Hdg}([j_S:S \hookrightarrow X])$, for $S \in \cS$. 
\end{remark}

In view of Corollary \ref{submodule} and Proposition \ref{ccs}, we get the following;
\begin{thm}\label{cohVconstant}
 Assume $[IC'^H_X]$ is virtually stratum-wise constant. Then $[IC'^H_X]$ belongs to the $K_0(\MHM(pt))=K_0(\MHS^p)$-submodule generated by $\im(\chi_{Hdg}) \subset K_0(\MHM(X))$. So if $X$ is a compact variety, then $[IC'^H_X]$ satisfies property (H), hence Conjecture \ref{BSY} holds for $X$.
\end{thm}

\begin{corollary}\label{ts}
Conjecture \ref{BSY} holds for:
\begin{itemize}
\item[(a)] matroid Schubert varieties.
\item[(b)] Schubert varieties in generalized flag varieties.
\item[(c)] any compact toric variety.
\end{itemize}
\end{corollary}

\begin{proof}
The first two cases follow from Example \ref{ex411}. 
If $X$ is a toric variety with torus $T$, it is well known that $IC_X$ has constant cohomology sheaves along the strata of the stratification consisting of the torus orbits, e.g., see \cite{MS2} and the references therein. In particular, $[IC'^H_X]$ is virtually stratum-wise constant. \end{proof}


\begin{remark}
    An important example of Corollary \ref{ts}(b) consists of  Schubert varieties in Grassmanians. In this case,  Conjecture \ref{BSY} is already known by results of the recent paper \cite{BSW}, where completely different methods (only working for Grassmanians) are used.
\end{remark}

Using Kleiman's generic transversality result (e.g., in the formulation of
\cite[Theorem 18]{SSW}), the following result can be applied to (iterated) generic translates of Schubert varieties in generalized flag varieties.

\begin{proposition}\label{generic}
Let $X_i\subset M$ for $i=1,2$ be two closed pure-dimensional algebraic subvarieties of a pure-dimensional complex algebraic manifold $M$, which are stratified transversal (with respect to suitable algebraic Whitney stratifications of $X_1$ and $X_2$), with $X:=X_1\cap X_2$ then also 
pure-dimensional. Assume $[IC'^H_{X_i}]$
 belongs to the $K_0(\MHM(pt))$-submodule generated by $\im(\chi_{Hdg}) \subset K_0(\MHM(X_i))$ for $i=1,2$. 
 
 Then also $[IC'^H_{X}]$
 belongs to the $K_0(\MHM(pt))$-submodule generated by $\im(\chi_{Hdg})$ 
 $ \subset K_0(\MHM(X))$. So if $X$ is a compact variety, then $[IC'^H_X]$ satisfies property (H), hence Conjecture \ref{BSY} holds for $X$.
\end{proposition}

\begin{proof}
Let $d: M\to M\times M$ be the diagonal embedding, so that $d$ is transversal  to $X_1\times X_2$ with respect to the product stratification,
with $X=d^{-1}(X_1\times X_2)$. Then
$$IC'^H_{X}\simeq d^*\left(IC'^H_{X_1}\boxtimes IC'^H_{X_2} \right) $$
(see, e.g., \cite[Example 10.2.34]{MS4} for the corresponding result of the underlying sheaf complexes), so that the claim follows from Remark \ref{module-functorial} and Corollary
\ref{submodule}.
\end{proof}

\begin{example}\label{Richardson}
Let $X$ be a {\it Richardson variety} (i.e., the transversal intersection of a Schubert variety with an opposite Schubert variety), or an {\it intersection variety} in the sense of \cite{BC} (i.e., the
intersection of a finite number of general translates of Schubert varieties)
in a generalized flag variety $M=G/P$. Then Conjecture \ref{BSY} holds for $X$.
\end{example}

We conclude this section by describing a general framework that guarantees that $[IC'^H_X]$ is virtually stratum-wise constant. Let $G$ be a linear algebraic group acting on a normal quasi-projective variety $X$. Then $X$ embeds $G$-equivariantly into an ambient complex algebraic $G$-manifold $Z$
(see e.g. \cite[Theorem 1]{Sum} for $G$ connected and 
\cite[Theorem 5.1.25]{CG} in general), and we can work with $G$-equivariant mixed Hodge modules on $Z$ with support on $X$, see \cite{A,T}. In particular, one can regard $IC^H_X$ as a $G$-equivariant intersection cohomology module (defined by the
equivariant intermediate extension of the constant Hodge module from the regular part of
$X$ to the ambient $G$-variety $Z$); compare with \cite{MS3}, where this setup is considered in the case when $G$ is a complex algebraic torus.
Assume $G$ acts on $X$ with finitely many orbits $S$, and
let $\cS$ be the corresponding complex algebraic stratification of $X$ so that for any orbit $S\in\cS$, $S$ is smooth, ${\bar S}\setminus S$ is a union of
orbits of lower dimension, and the sheaves $L_{S,\ell}:=\cH^{\ell}{IC_X}|_S$ are local systems on $S$ for any $\ell \in \bZ$. Assume moreover that for any $S\in \cS$, the stabilizer $G_s\subset G$ at a point $s\in S$ is connected. Then using \cite[Lemma 1.2]{T}, it follows that $IC'^H_X$
is stratum-wise constant, hence
$[IC'^H_X] \in K_0(\MHM(X))$ is virtually stratum-wise constant. This applies in particular to:
\begin{itemize}
    \item matroid Schubert varieties (e.g., see \cite[Corollary A.2]{Cr}).
    \item simply-connected spherical (scs) varieties in the sense of Brion-Joshua \cite{BJ}, which include both the toric varieties and the Schubert varieties considered in Corollary \ref{ts}. Indeed, stabilizers at points in the group orbits of these scs varieties are connected (cf. \cite[Lemma 3.6]{BJ}).
\end{itemize}
We thus have the following generalization of Corollary \ref{ts}:
\begin{corollary}\label{tsg}
Let $X$ be a normal complex projective variety with an action of a linear algebraic group $G$ with only finitely many $G$-orbits $S$, whose  stabilizer $G_s\subset G$ at a point $s\in S$ is connected for all $G$-orbits $S$.
Then Conjecture \ref{BSY} holds for  $X$.
\end{corollary}


\end{document}